\pdfoutput=1
\RequirePackage{ifpdf}
\ifpdf 
\documentclass[pdftex]{sigma}
\else
\documentclass{sigma}
\fi

\numberwithin{equation}{section}

\newtheorem{Theorem}{Theorem}[section]

\newtheorem{lem}[Theorem]{Lemma}
\newtheorem{prop}[Theorem]{Proposition}
 { \theoremstyle{definition}
\newtheorem{dfn}[Theorem]{Definition}

\newtheorem{rem}[Theorem]{Remark}
\newtheorem{assump}[Theorem]{Assumption}
}

\usepackage{mypackage_WRT} 
\usepackage{mycommand_WRT} 
\usetikzlibrary{knots}
\usetikzlibrary{positioning}

\newcommand{\trans}[1]{#1^\top}

\begin{document}
\allowdisplaybreaks

\newcommand{\arXivNumber}{2110.10958}

\renewcommand{\PaperNumber}{034}

\FirstPageHeading

\ShortArticleName{Witten--Reshetikhin--Turaev Invariants, Homological Blocks, and Quantum Modular Forms}

\ArticleName{Witten--Reshetikhin--Turaev Invariants, \\Homological Blocks, and Quantum Modular Forms\\ for~Unimodular Plumbing H-Graphs}

\Author{Akihito MORI and Yuya MURAKAMI}
\AuthorNameForHeading{A.~Mori and Y.~Murakami}
\Address{Mathematical Institute, Tohoku University, 6-3, Aoba, Aramaki, Aoba-Ku,\\
Sendai 980-8578, Japan}
\Email{\href{mailto:akihito.mori.s5@dc.tohoku.ac.jp}{akihito.mori.s5@dc.tohoku.ac.jp}, \href{mailto:yuya.murakami.s8@dc.tohoku.ac.jp}{yuya.murakami.s8@dc.tohoku.ac.jp}}
\URLaddress{
	\url{https://sites.google.com/view/yuya-murakami/home}}

\ArticleDates{Received November 23, 2021, in final form April 28, 2022; Published online May 07, 2022}

\Abstract{Gukov--Pei--Putrov--Vafa constructed $ q $-series invariants called homological blocks in a physical way in order to categorify Witten--Reshetikhin--Turaev (WRT) invariants and conjectured that radial limits of homological blocks are WRT invariants. In this paper, we prove their conjecture for unimodular H-graphs. As a consequence, it turns out that the WRT invariants of H-graphs yield quantum modular forms of depth two and of weight one with the quantum set $ \mathbb{Q} $. In the course of the proof of our main theorem, we first write the invariants as finite sums of rational functions. We second carry out a systematic study of weighted Gauss sums in order to give new vanishing results for them. Combining these results, we finally prove that the above conjecture holds for H-graphs.}


\Classification{57K31; 57K10; 57K16; 11F27; 11L05; 11T24}

\section{Introduction} \label{sec:intro}

Many authors have studied Witten--Reshetikhin--Turaev (WRT) invariants for 3-manifolds.
Law\-ren\-ce--Zagier~\cite[Theorem 1]{LZ} proved that the WRT invariant for the Poincar\'{e} homology sphere $\Sigma(2, 3, 5)$ coincides with a radial limit of a false theta function.
Hikami~\cite[Theorem 9]{H_Bries} gene\-ra\-lized it for Brieskorn homology spheres $\Sigma(p_{1}, p_{2}, p_{3})$.
Subsequently, Hikami~\cite{H_Lattice} expressed the WRT invariants for Seifert fibered homology spheres with four singular fibers as a sum of limiting values of Eichler integrals.
Furthermore, Hikami~\cite{H_Lattice2} generalised it for Seifert fibered homology spheres with $M$-singular fiberes.

These works allow one to expect that WRT invariants are radial limits for some $ q $-series with integer coefficients.

For plumbed 3-manifolds, Gukov--Pei--Putrov--Vafa gave a definition of a candidate, called homological blocks from a physical viewpoint, and conjectured that they categorify WRT invariants, that is, WRT invariants of any closed orientable $ 3 $-manifolds are linear combinations of radial limits of homological blocks~\cite[Conjecture 2.1]{GPPV}.

For Seifert fibered homology spheres, Fuji--Iwaki--Murakami--Terashima~\cite{FIMT} introduced $ q $-series called the WRT functions and proved that their radial limits are WRT invariants.
They are identified with homological blocks by Andersen--Mistegård~\cite{AM}.
Independently in a different way, Andersen--Mistegård~\cite[Theorem 4]{AM} proved that radial limits of homological blocks coincide with WRT invariants for Seifert fibered homology spheres.
To our knowledge, the precise relation between homological blocks and Eichler integrals for the cases treated in Hikami~\cite{H_Lattice2} is not known except for \cite{CCFGH,Ch,GMP, W}.

Seifert manifolds are special cases of plumbed manifolds.
Plumbed graphs yield Seifert manifolds if they have the only one vertex with the degree at least three.
H-graphs shown in Figure \ref{fig:H-graph} are the simplest graphs which are not such graphs.
It is expected that H-graphs yield non-Siefert manifolds.
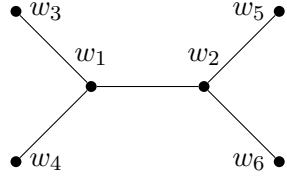
\begin{figure}[h]	\centering
		\begin{tikzpicture}
			\node[shape=circle,fill=black, scale = 0.4] (1) at (0,0) { };
			\node[shape=circle,fill=black, scale = 0.4] (2) at (1.5,0) { };
			\node[shape=circle,fill=black, scale = 0.4] (3) at (-1,-1) { };
			\node[shape=circle,fill=black, scale = 0.4] (4) at (-1,1) { };
			\node[shape=circle,fill=black, scale = 0.4] (5) at (2.5,1) { };
			\node[shape=circle,fill=black, scale = 0.4] (6) at (2.5,-1) { };
			
			\node[draw=none] (B1) at (0,0.4) {$ w_1 $};
			\node[draw=none] (B2) at (1.5, 0.4) {$ w_2 $};
			\node[draw=none] (B3) at (-0.6,1) {$ w_3 $};
			\node[draw=none] (B4) at (-0.6,-1) {$ w_4 $};
			\node[draw=none] (B5) at (2.1,1) {$ w_5 $};		
			\node[draw=none] (B6) at (2.1,-1) {$ w_6 $};	
			
			\path [-](1) edge node[left] {} (2);
			\path [-](1) edge node[left] {} (3);
			\path [-](1) edge node[left] {} (4);
			\path [-](2) edge node[left] {} (5);
			\path [-](2) edge node[left] {} (6);
		\end{tikzpicture}
		\caption{The H-graph $ \Gamma $.}\label{fig:H-graph}
	\end{figure}

In the case of H-graphs, Bringmann--Mahlburg--Milas~\cite{BMM} showed that radial limits of homological blocks coincide with sums of radial limits of multiple Eichler integrals.
Moreover, they proved that the radial limits of a homological block make up a quantum modular form of depth two and of weight one with the quantum set $ \mathbb{Q} $.

Meanwhile, for non-Seifert manifolds, it remains to be seen whether WRT invariants relate to homological blocks \cite[Conjecture~2.1]{GPPV} and have modularity.

In this paper, we study WRT invariants of H-graphs.
We will write them explicitly as weighted Gauss sums.
Moreover, we will prove that they relate to homological blocks and they yield quantum modular forms of depth two and of weight one with the quantum set $ \mathbb{Q} $.

Let us explain our setting.
Let $ \Gamma $ be the H-graph with vertices $ 1, \dots, 6 $ and weights $ w_v \in \Z_{<0} $ for each vertex $ v $ shown in Figure~\ref{fig:H-graph} such that its linking matrix
\begin{gather*}
W = \pmat{w_1 & 1 & 1 & 1 & 0 & 0 \\
	1 & w_2 & 0 & 0 & 1 & 1 \\
	1 & 0 & w_3 & 0 & 0 & 0 \\
	1 & 0 & 0 & w_4 & 0 & 0 \\
	0 & 1 & 0 & 0 & w_5 & 0 \\
	0 & 1 & 0 & 0 & 0 & w_6 }
\end{gather*}
is negative definite and its determinant is $ 1 $.

Then we obtain the link $ \calL(\Gamma) $ defined by $ \Gamma $ shown in Figure \ref{fig:surgery_diagram}.
Let $M_3(\Gamma)$ be the plumed 3-ma\-nifold obtained from $ S^3 $ through the surgery along the diagram $ \calL(\Gamma) $.
We obtain $H_1(M_{3}(\Gamma), \Z)\allowbreak \cong \Z^{6} / W\big(\Z^{6}\big)$ by use of the Mayer--Vietoris sequence.
Since $\det W = 1$, $M_{3}(\Gamma)$ is an integral homology sphere.
\begin{figure}[h]	\centering
\begin{tikzpicture}[scale=0.5]
			\begin{knot}[
				clip width=5,
				flip crossing=2,
				flip crossing=4,
				flip crossing=6,
				flip crossing=7,
				flip crossing=9
				]
				\strand[thick] (0, 0) circle [x radius=3cm, y radius=1.5cm];
				\strand[thick] (0, 0) +(4cm, 0pt) circle [x radius=3cm, y radius=1.5cm];
				\strand[thick] (0, 0) +(-2.5cm, 2cm) circle [radius=1.5cm];
				\strand[thick] (0, 0) +(-2.5cm, -2cm) circle [radius=1.5cm];
				\strand[thick] (0, 0) +(6.5cm, 2cm) circle [radius=1.5cm];
				\strand[thick] (0, 0) +(6.5cm, -2cm) circle [radius=1.5cm];
				
				\node (1) at (0, 2) {$w_{1}$};
				\node (2) at (4, 2) {$w_{2}$};
				\node (3) at (-2.5, 4) {$w_{3}$};
				\node (4) at (-2.5, -4) {$w_{4}$};
				\node (5) at (6.5, 4) {$w_{5}$};
				\node (6) at (6.5, -4) {$w_{6}$};
			\end{knot}
		\end{tikzpicture}
\caption{The surgery diagram $ \calL(\Gamma) $ corresponding to $\Gamma$.}\label{fig:surgery_diagram}
\end{figure}
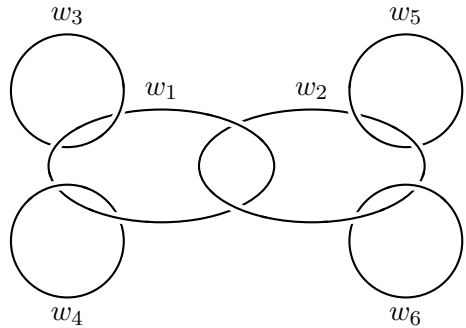

{\samepage
Such H-graphs yield Seifert manifolds if $ w_v = -1 $ for some $ 3 \le v \le 6 $.
Bringmann--Mahlburg--Milas~\cite[Theorem~1.2]{BMM} showed that there exist precisely $ 39 $ equivalence classes of such H-graphs with $ w_v \le -2 $ for each $ 3 \le v \le 6 $ up to graph isomorphism.

}

For a positive integer $k$, let $\tau_k$ be the WRT invariant normalised as $\tau_{k}\big(S^{3}\big) = 1$.
We study a~relation between the WRT invariant $\tau_{k}(M_{3}(\Gamma))$ and the homological block $\widehat{Z}_\Gamma (q)$.

For each complex number $ z $, we denote $ \bm{e}(z) := {\rm e}^{2\pi {\rm i} z} $.
For a point $ \tau \in \bbH $, let $ q := \bm{e}(\tau) $.
Gukov--Pei--Putrov--Vafa~\cite[formula~(3.145)]{GPPV} defined the homogical block $ \widehat{Z}(q) $ of the plumed graph.
In~our case, it can be written as
\begin{gather}
\widehat{Z}_\Gamma (q)		= \widehat{Z}_{\Gamma, \delta} (q)
= \widehat{Z}_\delta (q):=
q^{(18 - w_1 - \cdots - w_6)/4} \nonumber
\\
\hphantom{\widehat{Z}_\Gamma (q)= \widehat{Z}_{\Gamma, \delta} (q) = \widehat{Z}_\delta (q):=}
{}\times \mathrm{v.p.}\int_{\abs{z_v} = 1, 1 \le v \le 6}
\frac{\big(z_3 - z_3^{-1}\big)\big(z_4 - z_4^{-1}\big)\big(z_5 - z_5^{-1}\big)\big(z_6 - z_6^{-1}\big)}{\big(z_1 - z_1^{-1}\big)\big(z_2 - z_2^{-1}\big)}\nonumber\\
\hphantom{\widehat{Z}_\Gamma (q)= \widehat{Z}_{\Gamma, \delta} (q) = \widehat{Z}_\delta (q):=}
{}\times
\Theta_{-W, \delta} (q; z_1, \dots, z_6)
 \prod_{1 \le v \le 6} \frac{{\rm d}z_v}{2\pi {\rm i} z_v},
\label{eq:homological_block_for_H}
\end{gather}
where $ \delta := \trans{(1, 1, 1, 1, 1, 1)} $ is the signature of indexes of vertices of $ \Gamma $, $ \mathrm{v.p.} $ denotes the Cauchy principal value, and
\begin{gather*}
\Theta_{-W, \delta} (q; z_1, \dots, z_6) :=
\sum_{\bm{l} = \trans{(l_1, \dots, l_6)} \in 2\Z^6 + \delta} q^{-\trans{\bm{l}} W^{-1} \bm{l}/4} \prod_{1 \le v \le 6} z_v^{l_v}
\end{gather*}
is the theta function of the linking matrix~$ W $.

\cite[Conjecture 2.1 and relation~(A.28)]{GPPV} states a conjectural relation between the WRT invariants and radial limits of the homological blocks.
In our case, their relation is
\begin{equation} \label{eq:GPPV}
	\tau_k(M_3(\Gamma)) = \frac{1}{2\big(\zeta_{2k} - \zeta_{2k}^{-1}\big)}
	\lim_{\substack{q \to \zeta_k \\
			\abs{q} < 1	}
	} \widehat{Z}_\Gamma (q),
\end{equation}
where $ \zeta_k := {\rm e}^{2\pi {\rm i}/k} $.
In this paper, we prove this relation.

\begin{Theorem} \label{thm:WRT_homological_block_main}
	The notation being as above, $(\ref{eq:GPPV})$ holds.
\end{Theorem}

The proof of Theorem \ref{thm:WRT_homological_block_main} takes several steps as below:
\begin{enumerate}\itemsep=0pt
\setlength{\leftskip}{0.45cm}
	\item[$Step~1.$]
We give a formula (Proposition \ref{prop:WRT_inv}) for the WRT invariant $\tau_{k}(M_{3}(\Gamma))$ analogous to the formula given by Lawrence--Rozansky~\cite[formula~(4.2)]{LR} for any Seifert fibered manifold.
	This formula involves a certain type of quadratic forms of rank two which appear in false theta functions in~\cite[Proposition 5.3]{BMM} related to the homological blocks.
	\item[$Step~2.$] 
The sum involved in the formula (Proposition \ref{prop:WRT_inv}) is incomplete, that is, it is like a~form
	\begin{gather*}
	\sum_{n \in (\Z \smallsetminus k\Z)/kN\Z} f(n).
	\end{gather*}
	That means, we can not apply usual techniques in number theory.
	To avoid this situation, we interpolate the sum in the above formula for $\tau_{k}(M_{3}(\Gamma))$ by adding terms which diverge a priori.
	These terms turn out to be zero by vanishing results of ``weighted'' Gauss sums (Proposition \ref{prop:Gauss_sum}).
	\item[$Step~3.$] 
In Proposition \ref{prop:Phi_lim}, we prove that the radial limit of the homological block coincides with the interpolated formula of the WRT invariant by using the Euler--Maclaurin summation formula (Lemma \ref{lem:Euler-Maclaurin}) and vanishing results of weighted Gauss sums (Proposition \ref{prop:Gauss_sum}).
\end{enumerate}
As for our formula in \emph{Step}~1, we need to handle rank two quadratic forms which do not appear in Seifert case.
In \emph{Step}~2, we need some vanishing results for Gauss sums which we develop in this article.

We need more notation to explain the next main result.
Let $ S = ( l_{jk})_{1 \le j, k \le 2} $ be the submatrix of $ -W^{-1} = (l_{jk})_{1 \le j, k \le 6} $.
Let
\begin{gather*}
Q(m, n) := (m, n) S \trans{(m, n)}, \\
M := w_3 w_4, \\
N := w_5 w_6,
\\
a := -w_2 w_5 w_6 + w_5 + w_6, \\
c := -w_1 w_3 w_4 + w_3 + w_4.
\end{gather*}
Then we have $ Q(m, n) = Mam^2 - 2MNmn + Ncn^2 $ by a direct calculation and $ S^{-1}\big(\Z^2\big) = \frac{1}{M}\Z \oplus \frac{1}{N}\Z $ by Remark~\ref{rem:S_inv_factor}.

For non-zero integers $ w $ and $ w' $, we define a map $ \chi^{(w, w')} \colon \Z \to \Z/2ww'\Z \to \{ \pm1, 0 \} $ by
\begin{gather*}
\chi^{(w, w')}(n) :=
\begin{cases}
	ee' & \text{if} \ n \equiv w'e + we' +ww' \pmod{2ww'} \ \text{for some}\ e, e' \in \{ \pm 1 \},
\\
	0 & \text{otherwise}.
\end{cases}
\end{gather*}
We note that the generating function of this map appears in a calculation of the WRT invariant of the H-graph $ \Gamma $ (Proposition \ref{prop:WRT_inv}).
We also define a map
\[
\veps \colon \ (2S)^{-1}\big(\Z^2\big) \to (2S)^{-1}\big(\Z^2\big)/\Z^2 \to \{ \pm1, 0 \}
\]
 by
\[
 \veps(\alpha, \beta) := \chi^{(w_3, w_4)}(2M\alpha) \chi^{(w_5, w_6)}(2N\beta) .
 \]

Bringmann--Mahlburg--Milas~\cite[Theorems~1.1, 1.3 and~4.1]{BMM} shows that the homological block $\widehat{Z}_\Gamma (q)$ forms a depth two quantum modular form.
Combined with their results, Theorem \ref{thm:WRT_homological_block_main} implies the following theorem.

\begin{Theorem} 
	Let $ f_\Gamma \colon \Q \to \bbC $ be a map defined as
	\begin{gather*}
	f_\Gamma\bigg( \frac{h}{k} \bigg)
	= \frac{1}{k^2}
	\sum_{\gamma = \trans{(\alpha, \beta)} \in (2S)^{-1}(\Z^2) \cap [0, k)^2}
	\veps(\gamma)\, \bm{e}\bigg(\frac{h}{k}Q(\gamma) \bigg) \alpha \beta
	\end{gather*}
	for an irreducible fraction $ h/k \in \Q $.
	Then $ f_{\Gamma} $ is a quantum modular form of depth two and of weight one with the quantum set $ \Q $.
	Moreover, for any $ k \in \Z_{>0} $ we have
	\begin{gather*}
	f_\Gamma\left( \frac{1}{k} \right)
	=
	\frac{2 \left( \zeta_{2k}-\zeta_{2k}^{-1} \right)}{\zeta_{k}^{-(18 + \sum_{v=1}^{6}w_{v} + \sum_{v=3}^{6} 1/w_{v})/4}}
	\tau_k(M_3(\Gamma)).
	\end{gather*}
\end{Theorem}

This theorem yields new examples of quantum modular forms.
These examples have two interesting properties.
They have explicit expressions and WRT invariants appear in their special values.

This paper will be organised as follows.
In Section \ref{sec:pre}, we discuss WRT invariants, homological blocks, and quantum modular forms.
In Section \ref{sec:notation}, we set notation which we use throughout this paper.
We also define the rank two false theta function $ F_{Q, \veps}(\tau) $ which coincides with the homological block $ \widehat{Z}_\Gamma (q) $ without a simple factor.
In Section~\ref{sec:Gauss_sum}, we prove vanishing results for several types of Gauss sums, which we need to calculate a radial limit of $ F_{Q, \veps}(\tau) $ and $ \widehat{Z}_\Gamma (q) $ in Section \ref{sec:false_theta_lim}.
Finally, in Section \ref{sec:WRT}, we calculate the WRT invariant of $ M_3(\Gamma) $ and prove Theorem~\ref{thm:WRT_homological_block_main}.\looseness=1

\section{Preliminaries} \label{sec:pre}

In this section, we prepare basic materials.

\subsection{Witten--Reshetikhin--Turaev invariants} 

Reshetikhin--Turaev~\cite{RT} constructed the invariant $\tau_{k}$ for a pair $(M, L)$ of a 3-manifold $M$ and a~framed link $L$ in $M$.
This invariant is parametrised by a positive integer $k$.
In our situation \mbox{$L = \varnothing$}, so $\tau_{k}$ is an invariant for 3-manifolds.
This invariant is called the $\SU(2)$ Witten--Reshetikhin--Turaev (WRT) invariant.
Turaev gives a detailed description~\cite{R}.

For plumbed manifolds, Gukov--Pei--Putrov--Vafa caluculated the WRT invariant expli\-ci\-tly~\cite[formula~(A.10)]{GPPV}:
\begin{align*}
	\tau_{k}(M_{3}(\Gamma))
	={}
	&\frac{\bm{e}(-\sigma/8) \zeta_{k}^{3\sigma/4}}{2(2k)^{V/2} \big(\zeta_{2k}-\zeta_{2k}^{-1}\big)}
	\sum_{n \in (\Z^{V} / 2k\Z^{V}) \smallsetminus (k\Z^{V} / 2k\Z^{V})}
	\prod_{v=1}^{V}\zeta_{k}^{w_{v}(n_{v}^{2} - 1) / 4}
	\\
	&\times\bigg( \frac{1}{\zeta_{k}^{n_{v}/2} - \zeta_{k}^{-n_{v}/2}} \bigg)^{\deg(v)-2}
	\prod_{(v, v') \in \mathrm{Edges}} \frac{\zeta_{k}^{n_{v}n_{v'}/2} - \zeta_{k}^{-n_{v}n_{v'}/2}}{2},
\end{align*}
where $ b_+ $ and $ b_- $ are the number of positive and negative eigenvalues counted with multiplicities of the linking matrix of a plumbed graph $\Gamma$,
$\sigma := b_+ - b_- $ is the signature,
$V$ is the number of vertices of $\Gamma$,
$ \deg(v) $ is the degree for an $v$th vertex,
and $w_{v}$ is a weight for an $v$th vertex.
We~use this formula in Section \ref{sec:WRT}.

\subsection{Homological blocks} 

In this section, we show a definition of homological blocks and express them as false theta functions.
Though any mathematical definitions of these invariants are not given for general cases, they provided a rigorous definition for a particular case in which 3-manifolds are plumbed.

A plumbed graph is a tree whose vertices are weighted by integers.
A 3-manifold is said to be plumbed if it is obtained by the surgery along a plumbed graph.
\begin{dfn}[{\cite[Section 3.4]{GPPV}}]
	Let $G$ be a plumbed graph with vertices $1, \dots, N$ and $\delta$ be as in Section \ref{sec:intro}.
	For the plumbed $ 3 $-manifold $M_{3}(G)$ and $ b \in ( H_1(M_3(G), \Z) + \delta)/ \{ \pm1 \} \cong \big(2\Z^{N} / 2W\big(\Z^{N}\big) + \delta\big) / \{ \pm1 \} $, the homological block $\widehat{Z}_{G, b} (q)$ is defined as follows:
	\begin{gather*}
	\widehat{Z}_{G, b} (q)
	=
	q^{-\sum_{v = 1}^{N} (w_{v} + 3)/4}
	\mathrm{v.p.} \int_{\abs{z_{v}}=1, v =1, \dots, N} \prod_{v = 1}^{N} \frac{{\rm d}z_{v}}{2\pi {\rm i} z_{v}} (z_{v} - 1/z_{v})^{2 - \deg(v)}
	\Theta_{-W, b}(z),
	\end{gather*}
	where $w_{v}$ is the weight of the vertex $v$, $ \mathrm{v.p.} $ denotes the Cauchy principal value, $\deg(v)$ is the degree of $v$, $W$ is the linking matrix of $\calL(G)$, and
	\begin{gather*}
	\Theta_{-W, b} (q; z_1, \dots, z_N) :=
	\sum_{\bm{l} = \trans{(l_1, \dots, l_N)} \in 2W (\Z^N) + b} q^{-\trans{\bm{l}} W^{-1} \bm{l}/4} \prod_{1 \le v \le N} z_v^{l_v}.
	\end{gather*}
\end{dfn}

If $ G $ is an H-graph $ \Gamma $, the above definition of $ \widehat{Z}_{G, b} (q) $ coincides with (\ref{eq:homological_block_for_H}).
\begin{dfn} 
	For $ \tau \in \bbH$, we define false theta functions $F_{Q_{+}, \veps}(\tau)$ and $F_{Q_{-}, \veps}(\tau)$ by
	\begin{gather*}
	F_{Q_{+}, \veps}(\tau) := \sum_{0 \le \gamma \in (2S)^{-1} (\Z^2)} \veps(\gamma) q^{Q_+(\gamma)}, \qquad
	F_{Q_{-}, \veps}(\tau) := \sum_{0 \le \gamma \in (2S)^{-1} (\Z^2)} \veps(\gamma) q^{Q_-(\gamma)},
	\end{gather*}
	where $ Q_+(m, n) := Q(m, n)$, $Q_-(m, n) := Q(m, -n) $.
\end{dfn}

In order to calculate radial limits of the homological block $\widehat{Z}_\Gamma (q)$, we need an expression of~$\widehat{Z}_\Gamma (q)$ as a sum of false theta functions given by Bringmann--Mahlburg--Millas.

\begin{prop}[{\cite[Proposition 5.3]{BMM}}] 
	\begin{gather*}
	\widehat{Z}_\Gamma (q) = \frac{1}{2} q^{(-18 - w_1 - \cdots - w_6 - 1/w_3 - 1/w_4 - 1/w_5 - 1/w_6)/4}
	\left( F_{Q_{+}, \veps}(\tau) - F_{Q_{-}, \veps}(\tau) \right).
	\end{gather*}
\end{prop}

\subsection{Quantum modular forms} 

Zagier~\cite{Zagier_quantum} introduced quantum modular forms.
We describe its formal definition as follows.

\begin{dfn}[{\cite{Zagier_quantum}, \cite[Section 3.2]{BM}}]
	A map $ f \colon \calQ \to \bbC $ is a quantum modular form of weight $ k \in \frac{1}{2} \Z $
	for a subgroup $ \Gamma \subset \SL_2(\Z) $ (we assume $\Gamma \subset \Gamma_0(4)$ if $ k \in 1/2 + \Z$\,)
	and the quantum set $ \calQ \subset \Q $ if for any $ \gamma \in \Gamma $, the function $ f - \restrict{f}{k}\gamma $ can be extended to an open set of $ \R $ real analytically.
	
	We denote the $ \R $-vector space of such forms by $ \calQ_k(\Gamma) $.
\end{dfn}

Zagier~\cite[Examples 0--5]{Zagier_quantum} gave 6 examples of quantum modular forms.
Several authors studied quantum modular property of the WRT invariants for Seifert manifolds.
For example, Lawrence--Zagier~\cite{LZ} considered the Poincar\'{e} homology sphere $ \Sigma(2, 3, 5) $, Hikami~\cite[Proposition 8 and Theorem 9]{H_Bries} considered the Brieskorn homology sphere $ \Sigma(p_1, p_2, p_3) $ with $ p_1^{-1} + p_2^{-1} + p_3^{-1} < 1 $, and Hikami~\cite[Propositions 2, 4 and 8]{H_Seifert} considered spherical Seifert manifolds.
Bringmann--Rolen~\cite[Theorem 1.1]{BM} constructed quantum modular forms by special values of the $ L $-functions of cusp forms.

In our case, homological blocks are essentially rank two false theta functions and are not quantum modular forms.
They are depth two quantum modular forms defined by Bringmann--Kaszian--Milas~\cite[Definition 3]{BKM}.

\begin{dfn}[{\cite[Definition 3]{BKM}}]
	A map $ f \colon \calQ \to \bbC $ is a quantum modular form of depth two, weight $ k \in \frac{1}{2} \Z $
	for a subgroup $ \Gamma \subset \SL_2(\Z) $ $ ( \text{we assume $ \Gamma \subset \Gamma_0(4) $ if $ k \in 1/2 + \Z $}) $
	and the quantum set $ \calQ \subset \Q $ if for any $ \gamma \in \Gamma $, there exist $ k_1, \dots, k_r \in \frac{1}{2}\Z $ and an open subset $ R \subset \R $ such that
	\begin{gather*}
	f - \restrict{f}{k}\gamma \in \calO(R) \oplus \calQ_{k_1}(\Gamma) \calO(R) \oplus \cdots \oplus \calQ_{k_r}(\Gamma) \calO(R),
	\end{gather*}
	where $ \calO(R) $ is an $ \R $-vector space of real analytic functions on $ R $.
\end{dfn}

Bringmann--Mahlburg--Milas~\cite[Theorems 1.1, 1.3, 3.1 and 4.1]{BMM} proved the following theorem.

\begin{Theorem} \label{thm:BMM}
	Let $ S \in \Sym_2^+(\Z) $ be a positive definite symmetric matrix and $ Q(m, n) := (m, n) S \trans{(m, n)} = Am^2 + 2Bmn + Cn^2$ be the corresponding quadratic form such that \mbox{$ A, B, C \in \Z $}.
	Let $ \veps \colon (2S)^{-1}\big(\Z^2\big) \to (2S)^{-1}\big(\Z^2\big)/\Z^2 \to \bbC $ be a map such that
	$ \veps(\alpha, \beta) = \veps(1 - \alpha, \beta) = \veps(\alpha, 1 - \beta) = \veps(1 - \alpha, 1 - \beta) $.
	Let $ K \in \Z_{>0} $ and
	\begin{gather*}
	F_{K, Q, \veps}(\tau) := \sum_{0 \le \gamma \in (2S)^{-1}(\Z^2)} \veps(\gamma) q^{KQ(\gamma)}.
	\end{gather*}
	Then the followings hold:
	\begin{enumerate}
		\item[$(i)$] 
		{\rm (\cite[Theorems 1.1 and 3.1]{BMM})}
		The function $ F_{K, Q, \veps} $ is a quantum modular form of depth two and of weight one for the group $ \Gamma(8 \lcm(A, C) K \det S) $ with the quantum set
		\begin{gather*}
		\calQ_{K, Q, \veps} := \left\{\! \frac{h}{k} \in \Q \bigg|
		\gcd(h, k) = 1, \,
		k \in \Z_{>0}, \!
		\sum_{0 \le \gamma \in (2S)^{-1}(\Z^2)/k\Z^2}\!\!\! \veps(\gamma) \bm{e}\bigg( \frac{h}{k} KQ(\gamma) \bigg) = 0
		\!\right\}\!.
		\end{gather*}
		\item[$(ii)$] 
		{\rm (\cite[Theorems 1.3 and 4.1]{BMM})}
		Suppose that there exist positive integers $ M, N, a, c $ and an integer $ b $ such that $ A = Ma, B = 2MNb, C = Nc $ and $ ac - MNb^2 = 1 $.
		Let $ \alpha_1 \neq \alpha_2, \beta_1 \neq \beta_2 $ be rational numbers in the open interval $ (0, 1) $ such that the denominators of irreducble fraction expansions of $ \alpha_j $ and $ \beta_j $ are $ 2M $ and $ 2N $ respectively,
		$ \alpha_1^2 \equiv \alpha_2^2 \pmod{\frac{1}{2M}\Z} $, and $ \beta_1^2 \equiv \beta_2^2 \pmod{\frac{1}{2N}\Z} $.
		For a rational number $ \alpha $, let $ \calS(\alpha) := \{ \alpha, 1 - \alpha \} $.
		Let $ \calS_1 := \calS(\alpha_1) \times \calS(\beta_1) \cup \calS(\alpha_2) \times \calS(\beta_2) $,
		$\calS_2 := \calS(\alpha_1) \times \calS(\beta_2) \cup \calS(\alpha_2) \times \calS(\beta_1)$,
		$ \calS := \calS_1 \cup \calS_2 $, and
		$ \veps \colon (2S)^{-1}\big(\Z^2\big) \to (2S)^{-1}\big(\Z^2\big)/\Z^2 \to \bbC $ be a map such that
		\begin{gather*}
		\veps (\gamma) =
		\begin{cases}
			(-1)^{j+1} & \text{if} \ \gamma \in \calS_i + \Z^2,
\\
			0 & \text{if} \ \gamma \notin \calS + \Z^2.
		\end{cases}
		\end{gather*}
		Then we have $ \calQ_{2, Q, \veps} = \Q $.
		In particular, the function $ F_{2, Q, \veps} $ is a quantum modular form of depth two and of weight one for the group $ \Gamma(16MN \lcm(a, c) \det S) $ and the quantum set $ \Q $.
	\end{enumerate}
\end{Theorem}

Here we remark that in the setting in Theorem \ref{thm:BMM}$(ii)$, we have $ (2S)^{-1}\big(\Z^2\big) = \frac{1}{2M}\Z \oplus \frac{1}{2N}\Z $ by the same argument as in Remark \ref{rem:S_inv_factor}.
In Theorem \ref{thm:quantum_mod_form}, we will treat the quantum modularity of $ F_{Q, \veps} $ instead of $ F_{2Q, \veps} $ in Theorem \ref{thm:BMM}$(i)$.


\subsection{A reciprocity formula for Gauss sums}


To calculate the WRT invariant of $ M_3(\Gamma) $, we need the reciprocity formula in the following proposition.

\begin{prop}[{\cite[Theorem 1]{DT}}] \label{prop:reciprocity}
	Let $ L $ be a lattice of finite rank $ n $ equipped with a non-degenerated symmetric $ \Z $-valued bilinear form $ \sprod{\cdot, \cdot} $
	and let $ \sigma $ be the signature of the quadratic form $ \sprod{x, h(y)} $.
	We write
	\begin{gather*}
	L' := \{ y \in L \otimes \R \mid \sprod{x, y} \in \Z \text{ for all } x \in L \}
	\end{gather*}
	for the dual lattice.
	Let $ 0 < k \in \abs{L'/L} \Z$, $z \in \frac{1}{k} L $,
	and $ h \colon L \otimes \R \to L \otimes \R $ be a self-adjoint automorphism such that $ h(L') \subset L' $ and $ \frac{k}{2} \sprod{y, h(y)} \in \Z $ for all $ y \in L' $.
	Then it holds
	\begin{gather*}
\sqrt{\abs{L'/L}}\! \sum_{x \in L/kL}\!\! \bm{e} \bigg( \frac{1}{2k} \sprod{x, h(x)} + \sprod{x, z}\!\! \bigg)
= \frac{\bm{e}(\sigma/8) k^{n/2}}{\sqrt{\abs{\det h}}}
	\!\!	\sum_{y \in L'/h(L')}\!\! \bm{e} \bigg({-}\frac{k}{2} \sprod{y + z, h^{-1}(y + z)}\!\! \bigg).
	\end{gather*}
\end{prop}


\section{Notation} \label{sec:notation}


In this section, we prepare notation which we use throughout this paper.
The setting in Sec\-tion~\ref{sec:intro} is a particular case of the setting in this section.

We fix an integer $ h $ and a positive integer $ k $ such that $ \gcd(h, k) = 1 $.
Let $ M$, $N$, $a$, $c $ be positive integers and an integer $ b $ such that $ ac - MNb^2 = 1 $.
Let
\begin{gather*}
S = \pmat{Ma & MNb \\ MNb & Nc}
\end{gather*}
be a positive definite symmetric matrix.
The symbol $ \gamma = \trans{(\alpha, \beta)} $ denotes any element in $ \Q^2 $.
Here, $ t $ is the transpose of vectors.
Let $ Q(\alpha, \beta) := (\alpha, \beta) S \trans{(\alpha, \beta)} = Ma\alpha^2 + 2MNb \alpha \beta + Nc\beta^2 $ be the corresponding quadratic form.
Let $ \veps \colon (2S)^{-1}\big(\Z^2\big) \to (2S)^{-1}\big(\Z^2\big)/\Z^2 \to \bbC $ be a map.
For a complex number $ z $, we denote $ \bm{e}(z) := {\rm e}^{2\pi {\rm i} z} $.
Let $ \bbH $ be the upper half plane, that is, the set of complex numbers whose imaginary part are positive.
For a point $ \tau \in \bbH $, let {\samepage
\begin{gather*}
F_{Q, \veps}(\tau) := \sum_{0 \le \gamma \in (2S)^{-1}(\Z^2)} \veps(\gamma) q^{Q(\gamma)}
\end{gather*}
be a false theta function.}

We can rewrite a lattice $ S^{-1}\big(\Z^2\big) $ explicitly as below.
\begin{rem} \label{rem:S_inv_factor}
	Let
	\begin{gather*}
	A := \pmat{c & -Nb \\ -Mb & a} \in \SL_2(\Z).
	\end{gather*}
	Then we have
	\begin{gather*}
	SA = \pmat{M & 0 \\ 0 & N}
	\end{gather*}
	and $ S^{-1}(\Z^2) = \frac{1}{M}\Z \oplus \frac{1}{N}\Z $.
	Thus, we obtain $ \veps \colon \frac{1}{M}\Z/\Z \oplus \frac{1}{N}\Z/\Z \to \bbC $.
\end{rem}
\section{Vanishing results of Gauss sums} \label{sec:Gauss_sum}

In this section, we prove several claims which state vanishing results of weighted Gauss sums.
We~will use the results in this section to compute the WRT invariants and limit values of the homological blocks.
Let us keep the notation in Section \ref{sec:notation}, for example, $ h, k, Q(\gamma), \veps \colon (2S)^{-1}\big(\Z^2\big)/\allowbreak\Z^2 \to \bbC $.

\subsection{Main results in this section} 

Throughout this section, we argue that Gauss sums vanish in the setting in Section~\ref{sec:notation} and assume the following assumption for a map $ \veps \colon (2S)^{-1}\big(\Z^2\big)/\Z^2 \to \bbC $.

\begin{assump}\qquad \label{assump:Gauss_sum}
	\begin{enumerate}
		\item[$(i)$] 
		It holds
		\begin{gather*}
		\sum_{\gamma \in (2S)^{-1}(\Z^2)/\Z^2} \veps(\gamma) = 0.
		\end{gather*}
		\item[$(ii)$] 
		For $ \gamma = \trans{(\alpha, \beta)} \in (2S)^{-1}\big(\Z^2\big) $ such that $ \veps(\gamma) \neq 0 $, the denominator of the irreducible fractions of $ \alpha $ and $ \beta $ are $ 2M $ and $ 2N $ respectively.
		\item[$(iii)$] 
		For $ \gamma = \trans{(\alpha, \beta)} \in (2S)^{-1}\big(\Z^2\big) $ such that $ \veps(\gamma) \neq 0 $,
		\begin{gather*}
		M\alpha \pmod \Z, \qquad
		N\beta \pmod \Z, \qquad
		M\alpha^2 \pmod \Z,
\\
		N\beta^2 \pmod \Z, \qquad
		2MN\alpha \beta \pmod \Z
		\end{gather*}
		are independent of $ \gamma $.
	\end{enumerate}
\end{assump}

Here we remark that the map $ \veps \colon (2S)^{-1}\big(\Z^2\big)/\Z^2 \to \{ \pm1, 0 \} $ defined in Section \ref{sec:intro} satisfies the above assumption.

Our main results in this section are the following.

\begin{prop} \label{prop:Gauss_sum}
	Under Assumption $ \ref{assump:Gauss_sum} $, the followings are true:
	\begin{enumerate}
		\item[$(i)$] 
		It holds
\begin{gather*}
		\sum_{\gamma \in (2S)^{-1}(\Z^2)/k\Z^2} \veps(\gamma)
		\bm{e}\bigg( \frac{h}{k} Q(\gamma) \bigg)
		= 0.
		\end{gather*}
		\item[$(ii)$]
		We assume that there exist two maps $ \chi \colon \frac{1}{2M}\Z \to \frac{1}{2M}\Z/\Z \to \bbC $ and $ \psi \colon \frac{1}{2N}\Z \to \frac{1}{2N}\Z/\Z \to \bbC $ such that
		\begin{gather*}
		\sum_{\alpha \in \frac{1}{2M}\Z/\Z} \chi(\alpha) 		
		= \sum_{\beta \in \frac{1}{2N}\Z/\Z} \psi(\beta)
		= 0
		\end{gather*}
		and $ \veps\big(\trans{(\alpha, \beta)}\big) = \chi(\alpha) \psi(\beta) $ for any $ \alpha \in \frac{1}{2M}\Z/\Z $ and $ \beta \in \frac{1}{2N}\Z/\Z $.
		Then for any map $ C \colon \Q / k\Z \to \bbC $, we have
		\begin{gather*}
		\sum_{\gamma = \trans{(\alpha, \beta)} \in (2S)^{-1}(\Z^2)/k\Z^2} \veps(\gamma)
		\bm{e}\bigg(\frac{h}{k} Q(\gamma)\bigg) C(\alpha)
\\ \qquad
		{}=
		\sum_{\gamma = \trans{(\alpha, \beta)} \in (2S)^{-1}(\Z^2)/k\Z^2} \veps(\gamma)
		\bm{e}\bigg(\frac{h}{k} Q(\gamma)\bigg) C(\beta)
		= 0.
		\end{gather*}
		\item[$(iii)$] 
		Under the assumption in $(ii)$, let $ B \colon \Q / k\Z \to \bbC $ be a map such that
		\begin{gather*}
		\sum_{\alpha \in \frac{1}{2M}\Z/\Z} \chi(\alpha) \widetilde{B}(\alpha) = 0,
		\end{gather*}
		where $ \widetilde{B}(\alpha) := \sum_{0 \le m < k} B(\alpha + m) $.
		Then it holds
		\begin{gather*}
			\sum_{\mu \in k\Z \oplus \Z/2kS(\Z^2)} \bm{e}\bigg(\frac{1}{4k} {\trans{\mu}} S^{-1} \mu\bigg)
			\sum_{\gamma = \trans{(\alpha, \beta)} \in (2S)^{-1}(\Z^2)/k\Z^2} \veps(\gamma)
			\bm{e}\bigg(\frac{1}{k} {\trans{\gamma}} \mu\bigg) B(\alpha) C(\beta)
\\ \qquad
		{}	=
			\sum_{\mu \in k\Z^2/2kS(\Z^2)} \bm{e}\bigg(\frac{1}{4k} {\trans{\mu}} S^{-1} \mu\bigg)\!\!
			\sum_{\gamma = \trans{(\alpha, \beta)} \in (2S)^{-1}(\Z^2)/k\Z^2}\!\! \veps(\gamma)
			\bm{e}\bigg(\frac{1}{k} {\trans{\gamma}} \mu\bigg) B(\alpha) C(\beta)
\\ \qquad
			{}= 0.
		\end{gather*}
		\item[$(iv)$] 
		Under the assumption in $(ii)$, we have
		\begin{gather*}
			\sum_{\mu \in k\Z \oplus \Z/2kS(\Z^2)} \bm{e}\bigg(\frac{1}{4k} {\trans{\mu}} S^{-1} \mu\bigg)
			\sum_{\gamma = \trans{(\alpha, \beta)} \in (2S)^{-1}(\Z^2) \cap [0, k)^2} \veps(\gamma)
			\bm{e}\bigg(\frac{1}{k} {\trans{\gamma}} \mu\bigg) B_1(\alpha) B_1(\beta)
\\ \qquad
			{}=
			\sum_{\mu \in \Z \oplus k\Z/2kS(\Z^2)}\!\!\!\!\!\! \bm{e}\bigg(\frac{1}{4k} {\trans{\mu}} S^{-1} \mu\bigg)\!\!\!\!
			\sum_{\gamma = \trans{(\alpha, \beta)} \in (2S)^{-1}(\Z^2) \cap [0, k)^2} \!\!\!\!\!\veps(\gamma)
			\bm{e}\bigg(\frac{1}{k} {\trans{\gamma}} \mu\bigg) B_1(\alpha) B_1(\beta)
\\ \qquad
			{}=
			\sum_{\mu \in k\Z^2/2kS(\Z^2)}\!\!\!\! \bm{e}\bigg(\frac{1}{4k} {\trans{\mu}} S^{-1} \mu\bigg)\!\!\!\!
			\sum_{\gamma = \trans{(\alpha, \beta)} \in (2S)^{-1}(\Z^2) \cap [0, k)^2}\!\!\!\! \veps(\gamma)
			\bm{e}\bigg(\frac{1}{k} {\trans{\gamma}} \mu\bigg) B_1(\alpha) B_1(\beta)
\\ \qquad
			{}= 0,
		\end{gather*}
		where $ B_1(x) := x - 1/2 $ is the first Bernoulli polynomial.
		\item[$(v)$] 
		Under the assumption in $(ii)$, we have
		\begin{gather*}
			\sum_{\mu \in k\Z \oplus \Z/2kS(\Z^2)} \bm{e}\bigg(\frac{1}{4k} {\trans{\mu}} S^{-1} \mu\bigg)
			\sum_{\gamma = \trans{(\alpha, \beta)} \in (2S)^{-1}(\Z^2) \cap [0, k)^2} \veps(\gamma)
			\bm{e}\bigg(\frac{1}{k} {\trans{\gamma}} \mu\bigg) \alpha \beta
\\ \qquad
			{}=
			\sum_{\mu \in \Z \oplus k\Z/2kS(\Z^2)} \bm{e}\bigg(\frac{1}{4k} {\trans{\mu}} S^{-1} \mu\bigg)
			\sum_{\gamma = \trans{(\alpha, \beta)} \in (2S)^{-1}(\Z^2) \cap [0, k)^2} \veps(\gamma)
			\bm{e}\bigg(\frac{1}{k} {\trans{\gamma}} \mu\bigg) \alpha \beta
\\ \qquad
			{}=
			\sum_{\mu \in k\Z^2/2kS(\Z^2)} \bm{e}\bigg(\frac{1}{4k} {\trans{\mu}} S^{-1} \mu\bigg)
			\sum_{\gamma = \trans{(\alpha, \beta)} \in (2S)^{-1}(\Z^2) \cap [0, k)^2} \veps(\gamma)
			\bm{e}\bigg(\frac{1}{k} {\trans{\gamma}} \mu\bigg) \alpha \beta
\\ \qquad
			{}= 0.
	\end{gather*}
	\end{enumerate}
\end{prop}

Our proof of Proposition \ref{prop:Gauss_sum} is based on the proof of \cite[Theorem 4.1]{BMM}.
However, our assumptions are slightly different from it in \cite[Theorem 4.1]{BMM}.


\subsection[A proof of Proposition 4.2(i)]
{A proof of Proposition \ref{prop:Gauss_sum}$\boldsymbol{(i)}$}


Firstly, we prove Proposition \ref{prop:Gauss_sum}$(i)$.
To begin with, we need the result of the quadratic Gauss sum.
For integers $ a, b, c \in \Z $ with $ c > 0 $, we define the \textit{quadratic Gauss sum}
\begin{gather*}
G(a, b, c) := \sum_{n \in \Z/c\Z} \bm{e} \bigg(\frac{an^2 + bn}{c}\bigg).
\end{gather*}
The following result is elementary.

\begin{lem} \label{lem:Gauss_sum_fundamental}
	If $ \gcd(a, c) \nmid b $, then $ G(a, b, c) = 0 $.
\end{lem}

\begin{proof}
	Let $g := \gcd(a, c)$, $a' := a/g$, and $c' := c/g$.
	We have
	\begin{align*}
G(a, b, c)
&= \sum_{0 \le l < c'} \sum_{0 \le m < g} \bm{e} \bigg( \frac{ga'(l+c'm)^2 + b(l+c'm)}{gc'} \bigg)
\\
&= \sum_{0 \le l < c'} \bm{e} \bigg( \frac{ga' l^2 + bl}{gc'} \bigg) \sum_{0 \le m < g} \bm{e}
\bigg(\frac{bm}{g} \bigg).
	\end{align*}
	By assumption, the last sum vanishes.
\end{proof}

Proposition \ref{prop:Gauss_sum}$(i)$ follows from the following two lemmas.

\begin{lem} \label{lem:Gauss_sum_zero_general}
	If $ \gcd(M, k) > 1 $ or $ \gcd(N, k) > 1 $, then for $ \gamma = \trans{(\alpha, \beta)} \in \Q^2 $ we have
	\begin{gather*}
	\sum_{\mu \in \Z^2/k\Z^2} \bm{e}\bigg( \frac{h}{k} Q(\mu + \gamma) \bigg)
	= 0.
	\end{gather*}
\end{lem}

\begin{lem} \label{lem:Gauss_sum_indep_general}
	If $ \gcd(M, k) = \gcd(N, k) = 1 $, then
	\begin{gather*}
	\sum_{\mu \in \Z^2/k\Z^2} \bm{e}\bigg( \frac{h}{k} Q(\mu + \gamma) \bigg)
	\end{gather*}
	is independent of $ \gamma \in (2S)^{-1}\big(\Z^2\big) $ such that $ \veps(\gamma) \neq 0 $.
\end{lem}

\begin{proof}[Proof of Lemma $ \ref{lem:Gauss_sum_zero_general} $]
	We give a proof for the case when $ \gcd(M, k) > 1 $.
	Another case is similarly handled.
	We have
	\begin{align*}
		\sum_{\mu \in \Z^2/k\Z^2} \bm{e}\bigg( \frac{h}{k} Q(\mu + \gamma) \bigg)
		&= \sum_{\mu = \trans{(m, n)} \in \Z^2/k\Z^2} \bm{e}\bigg( \frac{h}{k} \big( Q(\mu) + 2\trans{\mu} S \gamma + Q(\gamma) \big) \bigg) \\
		&= \bm{e}\bigg(\frac{h}{k} Q(\gamma)\bigg)
		\sum_{n \in \Z/k\Z} \bm{e}\bigg( \frac{h}{k} \big( N c n^2 + 2Nn(Mb \alpha + c \beta) \big) \bigg) G_n,
	\end{align*}
	where $ r := 2M \alpha, s := 2N \beta \in \Z $, and $ G_n := G(hMa, h(ar + Mb(s + 2Nn)), k) $.
	Let $ g := \gcd(hMa, k) $.
	We have $ G_n = 0 $ if $ g \nmid h(ar + Mb(s + 2Nn)) $ by Lemma \ref{lem:Gauss_sum_fundamental}.
	Since $ g' := \gcd(M, k) $ divides $ g $, it suffices to show $ g \nmid h(ar + Mb(s + 2Nn)) $.
	
	We have $ \gcd(g', a) = \gcd(g', r) = \gcd(g', h) = 1 $ since $ ac - MN b^2 = 1 $ and $ \gcd(r, 2M) = \gcd(h, k) = 1 $ respectively.
	Thus, it holds $ \gcd(g, har) = 1 $.
	Therefore, we obtain
	\begin{gather*}
	h(ar + Mb(s + 2Nn)) \equiv ar \not\equiv 0 \pmod{g'}.\tag*{\qed}
\end{gather*}
\renewcommand{\qed}{}
\end{proof}

\begin{proof}[Proof of Lemma $ \ref{lem:Gauss_sum_indep_general} $]
	Since $ \gcd(M, k) = \gcd(N, k) = 1 $, there exist integers $ M^*, N^* \in \Z $ such that $ MM^* \equiv NN^* \equiv 1 \pmod k $.
	For $ \gamma \in (2S)^{-1}\big(\Z^2\big) $ such that $ \veps(\gamma) \neq 0 $,
	let $ \gamma^* := \trans{(MM^* \alpha, NN^* \beta)} \in \frac{1}{2} \Z^2 $.
	Since
	\begin{gather*}
	2 S (\gamma - \gamma^*)
	= 2 \pmat{Ma & MNb \\ MNb & Nc} \pmat{(1 - MM^*) \alpha \\ (1 - NN^*) \beta}
	\in k\Z^2,
	\end{gather*}
	we have
	\begin{gather*}
	Q(\mu + \gamma) - Q(\mu + \gamma^*)
	= Q(\gamma) - Q(\gamma^*) + 2 \trans{\mu} S (\gamma - \gamma^*)
	\equiv Q(\gamma) - Q(\gamma^*) \pmod{k\Z}.
	\end{gather*}
	Thus, we obtain
	\begin{align*}
		\sum_{\mu \in \Z^2/k\Z^2} \bm{e}\bigg( \frac{h}{k} Q(\mu + \gamma) \bigg)
		&= \sum_{\mu \in \Z^2/k\Z^2} \bm{e}\bigg( \frac{h}{k} ( Q(\mu + \gamma^*) + Q(\gamma) - Q(\gamma^*) ) \bigg)
\\
		&= \bm{e}\bigg( \frac{h}{k} ( Q(\gamma) - Q(\gamma^*)) \bigg)
		\sum_{\mu \in \Z^2/k\Z^2} \bm{e}\bigg( \frac{h}{k} Q(\mu) \bigg).
	\end{align*}
	By Assumption \ref{assump:Gauss_sum}$(iii)$, $ Q(\gamma) - Q(\gamma^*)\! \pmod{k\Z} $ is independent of $ \gamma $.
	Since $ \trans{(M\alpha, N\beta)} \!\!\pmod{\Z^2} $ is also independent of $ \gamma $ by Assumption \ref{assump:Gauss_sum}$(iii)$, there exists $ \delta_0 \in \frac{1}{2} \Z^2/\Z^2 $ which is independent of $ \gamma $ and $ \gamma^* \equiv \delta_0 \pmod{\Z^2} $.
	Thus,
	\begin{gather*}
	\sum_{\mu \in \Z^2/k\Z^2} \bm{e}\bigg( \frac{h}{k} Q(\mu + \gamma^*) \bigg)
	= \sum_{\mu \in \Z^2/k\Z^2} \bm{e}\bigg( \frac{h}{k} Q(\mu + \delta_0) \bigg)	
	\end{gather*}
	is also independent of $ \gamma $.
\end{proof}

\subsection[A proof of Proposition 4.2(ii)]{A proof of Proposition \ref{prop:Gauss_sum}$\boldsymbol{(ii)}$} 

Secondly, we prove Proposition \ref{prop:Gauss_sum}$(ii)$.

\begin{proof}[Proof of Proposition \ref{prop:Gauss_sum}$\boldsymbol{(ii)}$]
	We only prove the second equality.
	By assumption, we have
	\begin{gather*}
		\sum_{\gamma = \trans{(\alpha, \beta)} \in (2S)^{-1}(\Z^2)/k\Z^2} \veps(\gamma)
		\bm{e}\bigg(\frac{h}{k} Q(\gamma)\bigg) C(\beta)
\\ \qquad
		{}= \sum_{\beta \in \frac{1}{2N}\Z/k\Z} \psi(\beta)
		\bm{e}\bigg( \frac{h}{k} c \beta^2 \bigg) C(\beta)
		\sum_{\alpha \in \frac{1}{2M}\Z/k\Z} \chi(\alpha)
		\bm{e}\bigg( \frac{h}{k} M Q_{2N\beta}^{} (\alpha) \bigg),
	\end{gather*}
	where $ Q_s(\alpha) := a \alpha^2 + bs \alpha $.
	Then the second sum in the right hand side vanishes by the following lemma.
\end{proof}

\begin{lem} \label{lem:Gauss_sum_Bernoulli_single}
	Under the same assumption as in Proposition $\ref{prop:Gauss_sum}(ii)$, for a quadratic form $ Q_0(\alpha)\allowbreak = a_0 \alpha^2 + b_0 \alpha $ such that $ a_0, b_0 \in \Z, \gcd(M, a_0) = 1 $ and a map $ \widetilde{B} \colon \frac{1}{2M}\Z/\Z \to \bbC $ such that
	\begin{gather*}
	\sum_{\alpha \in \frac{1}{2M}\Z/\Z} \chi(\alpha) \widetilde{B}(\alpha) = 0,
	\end{gather*}
	it holds
	\begin{gather*}
	\sum_{\alpha \in \frac{1}{2M}\Z/k\Z} \chi(\alpha) \bm{e}\bigg( \frac{h}{k} M Q_0 (\alpha) \bigg) \widetilde{B}(\alpha) = 0.
	\end{gather*}
\end{lem}

Lemma \ref{lem:Gauss_sum_Bernoulli_single} follows from the following two lemmas.

\begin{lem} \label{lem:Gauss_sum_zero_Bernoulli}
	If $ \gcd(M, k) > 1 $, then for $ \alpha \in \frac{1}{2M} \Z $ such that $ \chi(\alpha) \neq 0 $, we have
	\begin{gather*}
	\sum_{m \in \Z/k\Z} \bm{e}\bigg( \frac{h}{k} M Q_0 (m + \alpha) \bigg)
	= 0.
	\end{gather*}
\end{lem}

\begin{lem} \label{lem:Gauss_sum_indep_Bernoulli}
	If $ \gcd(M, k) = 1 $, then
	\begin{gather*}
	\sum_{m \in \Z/k\Z} \bm{e}\bigg( \frac{h}{k} M Q_0 (m + \alpha ) \bigg)
	\end{gather*}
	is independent of $ \alpha \in \frac{1}{2M} \Z $ such that $ \chi(\alpha) \neq 0 $.
\end{lem}

\begin{proof}[Proof of Lemma $ \ref{lem:Gauss_sum_zero_Bernoulli} $]
	Let $ r := 2M\alpha $.
	We can write
	\begin{gather*}
	\sum_{m \in \Z/k\Z}\!\!\!\! \bm{e}\bigg( \frac{h}{k} M Q_0 \bigg(\!m \!+\! \frac{r}{2M} \bigg) \!\bigg)
\!= \bm{e}\bigg( \frac{h(a_0 r^2 \!+\! 2M b_0r)}{4Mk} \bigg)
\!\!\!	\sum_{m \in \Z/k\Z}\!\!\!\! \bm{e}\bigg( \frac{h(Ma_0 m^2 \!+\! (a_0r \!+\! Mb_0)m)}{k} \bigg).
	\end{gather*}
	Since $ \gcd(hMa, k) \mid \gcd(M, k) $ and $ \gcd(M, k) \nmid a_0 r + Mb_0 $, we obtain the claim.
\end{proof}

\begin{proof}[Proof of Lemma $ \ref{lem:Gauss_sum_indep_Bernoulli} $]
Since $ \gcd(M, k) = 1 $, there exists an integer $ M^* \in \Z $ such that $ MM^* \equiv 1 \pmod k $.
	For any integer $ m \in \Z $, we have
	\begin{gather*}
	Q_0 ( m + \alpha ) - Q_0 ( m + MM^* \alpha )
\\ \qquad
	{}	\equiv (2a_0 m + b)(1 - M M^*) \alpha + \big(1 - (MM^*)^2\big) a_0 \alpha^2 \pmod{k\Z}.
	\end{gather*}
	Thus, we obtain
	\begin{align*}
		\sum_{m \in \Z/k\Z}\!\! \bm{e}\bigg( \frac{h}{k} M Q_0 (m + \alpha ) \bigg)
		= {}
		&\bm{e}\bigg( \frac{h(1 - (MM^*)^2)}{k} a_0 M \alpha^2 + \frac{h(1 - M M^*)}{k} (2a_0 m + b) M
\alpha \bigg)
\\
		&\times \sum_{m \in \Z/k\Z} \bm{e}\bigg( \frac{h}{k} M Q_0 (m + MM^* \alpha ) \bigg).
	\end{align*}
	This sum is independent of $ \alpha $ by Assumption \ref{assump:Gauss_sum}$(iii)$.
\end{proof}

\subsection[A proof of Proposition 4.2(iii), (iv) and (v)]
{A proof of Proposition \ref{prop:Gauss_sum}$\boldsymbol{(iii)}$, $\boldsymbol{(iv)}$ and $\boldsymbol{(v)}$} 

Finally, we prove Proposition \ref{prop:Gauss_sum}$(iii)$, $(iv)$ and $(v)$.

\begin{proof}[Proof of Proposition \ref{prop:Gauss_sum}${\boldsymbol{(iii)}}$, $\boldsymbol{(iv)}$ and $\boldsymbol{(v)}$]
	Firstly, we prove Proposition~\ref{prop:Gauss_sum}$(iii)$.
	We on\-ly prove that the most left hand side vanishes.
	It can be written as
	\begin{gather*}
		\sum_{\beta \in \frac{1}{2N}\Z/k\Z}\!\!\!\! \psi(\beta) C(\beta)\!\!\!
		\sum_{n \in \Z/2kN\Z}\!\!\! \!\bm{e}\bigg( \frac{an^2}{4kN} + \frac{\beta n}{k} \bigg)
	\!\!\!\sum_{\alpha \in \frac{1}{2M}\Z/\Z}\!\!\! \!\chi(\alpha) \widetilde{B}(\alpha)\!\!\!
\sum_{m \in \Z/2M\Z}\!\!\!\!\bm{e}\bigg( \frac{kcm^2}{4M} - \frac{1}{2}bmn + \alpha m \bigg).
	\end{gather*}
	It suffices to show that the second double sum for $ \alpha $ and $ m $ vanishes.
	By Proposition \ref{prop:reciprocity}, it equals to
	{\samepage\begin{gather*}
\frac{2M{\rm i}}{\sqrt{kc}}
		\sum_{\alpha \in \frac{1}{2M}\Z/\Z} \chi(\alpha) \widetilde{B}(\alpha)
		\sum_{m \in \Z/kc\Z} \bm{e}\bigg({-}\frac{M}{kc} \bigg( m + \alpha - \frac{bn}{2} \bigg)^2 \bigg)
\\ \qquad
	{}=	\frac{2M{\rm i}}{\sqrt{kc}}
		\sum_{\alpha \in \frac{1}{2M}\Z/kc\Z} \chi\bigg(\alpha + \frac{Nbn}{2M} \bigg) \widetilde{B}\bigg(\alpha + \frac{Nbn}{2M} \bigg)
		\bm{e}\bigg({-}\frac{M}{kc} \alpha^2 \bigg).
	\end{gather*}
	This vanishes by Lemma \ref{lem:Gauss_sum_Bernoulli_single}.

}
	
	Secondly, we prove Proposition \ref{prop:Gauss_sum}$(iv)$.
	By the multiplication theorem of Bernoulli polynomials, we have
	\begin{gather*}
	B_1(\alpha) = \sum_{0 \le m < k} B\bigg( \frac{\alpha}{k} + m \bigg).
	\end{gather*}
	Since
	\begin{gather*}
	\sum_{\alpha \in \frac{1}{2M}\Z \cap [0, 1)} \chi(\alpha) B_1(\alpha)
	= \sum_{\alpha \in \frac{1}{2M}\Z \cap [0, 1)} \chi(\alpha) B_1(1-\alpha)
	= -\sum_{\alpha \in \frac{1}{2M}\Z \cap [0, 1)} \chi(\alpha) B_1(\alpha)
	\end{gather*}
	is equal to $ 0 $, we obtain the claim by Proposition \ref{prop:Gauss_sum}$(iii)$.
	
	Finally, we prove Proposition \ref{prop:Gauss_sum}$(v)$.
	By \ref{assump:Gauss_sum}$(i)$, we have
	\begin{align*}
		\sum_{\alpha \in \frac{1}{2M}\Z \cap [0, k)} \chi(\alpha) \alpha		&=
\sum_{\alpha \in \frac{1}{2M}\Z \cap [0, 1)} \chi(\alpha) \bigg( k\alpha + \frac{1}{2}k(k-1) \bigg)
		= 		k \sum_{\alpha \in \frac{1}{2M}\Z \cap [0, 1)} \chi(\alpha) \alpha
\\
		&= k \sum_{\alpha \in \frac{1}{2M}\Z \cap [0, 1)} \chi(\alpha) ( 1 - \alpha )
		= -k \sum_{\alpha \in \frac{1}{2M}\Z \cap [0, 1)} \chi(\alpha) \alpha.
	\end{align*}
	Since it vanishes, we obtain the claim by Proposition \ref{prop:Gauss_sum}$(iii)$.
\end{proof}


\section{A limit value of a false theta function} \label{sec:false_theta_lim}


In this section, we will calculate a radial limit of the false theta function $ F_{Q, \veps} (\tau) $.
We use notation in Section \ref{sec:notation} throughout this section.
A key fact is the following lemma which follows from the Euler--Maclaurin summation formula.

\begin{lem}[{\cite[Lemma 2.2]{BMM}}] \label{lem:Euler-Maclaurin}
	For two real numbers $ \alpha $ and $ \beta $ and $ C^\infty $-function $ f \colon \R^2 \to \bbC $ which has rapid decay, we have an asymptotic evaluation
	\begin{gather*}
	\sum_{m, n = 0}^{\infty} f(t(m+\alpha, n+\beta))
	\sim \sum_{j,l = -1}^{\infty} \frac{B_{j+1}(\alpha)}{(j+1)!} \frac{B_{l+1}(\beta)}{(l+1)!} f^{(j, l)}(0, 0) t^{j+l}
	\end{gather*}
	as $ t \to +0 $, where $ F(t) \sim G(t) $ means that $ F(t) = G(t) + O\big(t^N\big) $ for any positive integer $ N $,
	$ B_n(x) $ is the $ n $-th Bernoulli polynomial, and for integers $ j, l \ge 0 $ let
	\begin{gather*}
		f^{(-1, -1)}(0, 0) := \int_0^\infty \int_0^\infty f(x, y)\, {\rm d}x {\rm d}y, \\
		f^{(j, -1)}(0, 0) := -\int_0^\infty \frac{\partial^j f}{\partial x^j} (0, y)\, {\rm d}y, \\
		f^{(-1, l)}(0, 0) := -\int_0^\infty \frac{\partial^l f}{\partial y^l} (x, 0) \,{\rm d}x, \\
		f^{(j, l)}(0, 0) := \frac{\partial^{j+l} f}{\partial x^j \partial y^l} (0, 0).
	\end{gather*}
\end{lem}

We can calculate a radial limit of the false theta function $ F_{Q, \veps} (\tau) $ as follows.

\begin{lem} \label{lem:Phi_asymptotic}
	If
	\begin{gather*}
	\veps(\gamma) = \veps\big(\trans{(1, 1)} - \gamma\big) = \veps\big(\trans{(1 - \alpha, \beta)}\big) = \veps\big(\trans{(\alpha, 1 - \beta)}\big)
	\end{gather*}
	for any $ \gamma = \trans{(\alpha, \beta)} \in (2S)^{-1}\big(\Z^2\big) $, then we have an asymptotic evaluation
	\begin{equation} \label{eq:Phi_asymptotic}
		F_{Q, \veps} \bigg( \frac{h}{k} + \frac{t{\rm i}}{2\pi} \bigg)
		\sim \sum_{r = -1}^{\infty} a_{h, k}(r) t^r
	\end{equation}
	as $ t \to +0 $, where $ f(x, y) := {\rm e}^{-Q(x, y)} $ and
	\begin{align*}
		a_{h, k}(r) :={}& k^{2r}
\sum_{\gamma = \trans{(\alpha, \beta)} \in (2S)^{-1}(\Z^2) \cap [0, k)^2} \veps(\gamma) \bm{e}\bigg(\frac{h}{k} Q(\gamma)\bigg)
\\
		&\times\sum_{j+l = 2r, \, j, l \ge -1} \frac{1}{(j+1)! (l+1)!} B_{j+1}\bigg( \frac{\alpha}{k} \bigg) B_{l+1}\bigg( \frac{\beta}{k} \bigg) f^{(j, l)}(0, 0).
	\end{align*}
	
	Moreover, under Assumption $ \ref{assump:Gauss_sum} $ we have
	\begin{align*}
		a_{h, k}(r) ={}& k^{2r}
\sum_{\gamma = \trans{(\alpha, \beta)} \in (2S)^{-1}(\Z^2) \cap [0, k)^2} \veps(\gamma) \bm{e}\bigg(\frac{h}{k} Q(\gamma)\bigg) \\
		&{}\times\sum_{j+l = 2r, \, j, l \ge 0} \frac{1}{(j+1)! (l+1)!} B_{j+1}\bigg( \frac{\alpha}{k} \bigg) B_{l+1}\bigg( \frac{\beta}{k} \bigg) f^{(j, l)}(0, 0).
	\end{align*}
\end{lem}

\begin{proof}
	Our proof is based on the proof of \cite[Lemma 3.2]{BMM}.
	Since we can write
	\begin{gather*}
	F_{Q, \veps} \bigg( \frac{h}{k} + \frac{t{\rm i}}{2\pi} \bigg)
	=
	\sum_{\gamma = \trans{(\alpha, \beta)} \in (2S)^{-1}(\Z^2) \cap [0, k)^2} \veps(\gamma) \bm{e}\bigg(\frac{h}{k} Q(\gamma)\bigg)
	\sum_{\mu \in \Z_{\ge 0}^2} f(\sqrt{t}(\mu + \gamma)),
	\end{gather*}
	we have an asymptotic evaluation
	\begin{align*}
		F_{Q, \veps} \bigg( \frac{h}{k} + \frac{t{\rm i}}{2\pi} \bigg)
\sim{} &\sum_{j,l = -1}^{\infty} k^{j+l} t^{(j+l)/2}
\sum_{\gamma = \trans{(\alpha, \beta)} \in (2S)^{-1}(\Z^2) \cap [0, k)^2} \veps(\gamma) \bm{e}\bigg(\frac{h}{k} Q(\gamma)\bigg)
\\
		&\times\frac{1}{(j+1)! (l+1)!} B_{j+1}\bigg( \frac{\alpha}{k} \bigg) B_{l+1}\bigg( \frac{\beta}{k} \bigg) f^{(j, l)}(0, 0)
	\end{align*}
	by Lemma \ref{lem:Euler-Maclaurin}.
	By assumption, the right hand side yields
	\begin{gather*}
		\sum_{j,l = -1}^{\infty} k^{j+l} t^{(j+l)/2}
		\sum_{\gamma = \trans{(\alpha, \beta)} \in (2S)^{-1}(\Z^2) \cap [0, k)^2} \veps(\gamma) \bm{e}\bigg( \frac{h}{k} Q(k - \alpha, k - \beta) \bigg)
\\ \qquad
		{}\times\frac{1}{(j+1)! (l+1)!} B_{j+1}\bigg( \frac{k - \alpha}{k} \bigg) B_{l+1}\bigg( \frac{k - \beta}{k} \bigg) f^{(j, l)}(0, 0).
	\end{gather*}
	Since $ Q(k - \alpha, k - \beta) \equiv Q(\gamma) \pmod{k\Z} $ and $ B_n(1-x) = (-1)^n B_n(x) $, this becomes
	\begin{gather*}
		\sum_{j,l = -1}^{\infty} k^{j+l} t^{(j+l)/2} (-1)^{j+l+2}
		\sum_{\gamma = \trans{(\alpha, \beta)} \in (2S)^{-1}(\Z^2) \cap [0, k)^2} \veps(\gamma) \bm{e}\bigg(\frac{h}{k} Q(\gamma)\bigg)
\\ \qquad
		{}\times\frac{1}{(j+1)! (l+1)!} B_{j+1}\bigg( \frac{\alpha}{k} \bigg) B_{l+1}\bigg( \frac{\beta}{k} \bigg) f^{(j, l)}(0, 0).
	\end{gather*}
	Since it vanishes for the odd values of $ j+l $, we obtain the asymptotic evaluation (\ref{eq:Phi_asymptotic}).
	
	The latter claim follows from Proposition \ref{prop:Gauss_sum}$(i)$.
\end{proof}

Thus, we obtain the following proposition by Proposition \ref{prop:Gauss_sum}$(i)$ and $(ii)$.

\begin{prop} \label{prop:Phi_lim}
	If Assumption $ \ref{assump:Gauss_sum} $ and the assumption in Lemma $ \ref{lem:Phi_asymptotic} $ and Proposition $ \ref{prop:Gauss_sum}(ii)$ hold, then we have\vspace{-.5ex}
	\begin{gather*}
		\lim_{t \to 0} F_{Q, \veps} \bigg( \frac{h}{k} + ti \bigg)
		= \frac{1}{k^2}
		\sum_{\gamma = \trans{(\alpha, \beta)} \in (2S)^{-1}(\Z^2) \cap [0, k)^2} \veps(\gamma) \bm{e}\bigg(\frac{h}{k} Q(\gamma)\bigg)
		\alpha \beta.
	\end{gather*}
\end{prop}

We obtain the following quantum modular form by this proposition and Theorem \ref{thm:BMM}$(i)$ which is proved in \cite[Theorems 1.1 and 3.1]{BMM}.

\begin{Theorem} \label{thm:quantum_mod_form}
	Let $ f_\Gamma \colon \Q \to \bbC $ be a map defined by\vspace{-.5ex}
	\begin{gather*}
	f_\Gamma\bigg( \frac{h}{k} \bigg)
	= \frac{1}{k^2}
	\sum_{\gamma = \trans{(\alpha, \beta)} \in (2S)^{-1}(\Z^2) \cap [0, k)^2}
	\veps(\gamma) \bm{e}\bigg( \frac{h}{k}Q(\gamma) \bigg) \alpha \beta
	\end{gather*}
	for an irreducible fraction $ h/k \in \Q $.
	If Assumption $ \ref{assump:Gauss_sum} $ and the assumption in Lemma $ \ref{lem:Phi_asymptotic} $ and Proposition $ \ref{prop:Gauss_sum}(ii)$ hold, then $ f_{\Gamma} $ is a quantum modular form of depth two and of weight one with the quantum set $ \Q $.\vspace{-1ex}
\end{Theorem}

\section{Witten--Reshetikhin--Turaev invariants} \label{sec:WRT}

In this section, we prove that the WRT invariant is the radial limit of the homological block.
Firstly, we write the WRT invariant as a sum of rational functions.
Secondly, the Euler--Maclaurin summation formula is used to express the WRT invariant as a sum of Bernoulli polynomials.
Finally, we write the WRT invariant as the radial limit of the homological block calculated in Section~\ref{sec:false_theta_lim}.

We start with an expression for the WRT invariant given by Gukov--Pei--Putrov--Vafa~\cite[formula~(A.10)]{GPPV}.
In our situation it is written down as follows:
\begin{align}
		\tau_{k}(M_{3}(\Gamma))
		={}
		&\frac{-{\rm i} \zeta_{k}^{-9/2}}{16k^{3} \big(\zeta_{2k}-\zeta_{2k}^{-1}\big)}
		\sum_{l \in (\Z \smallsetminus k\Z)^{6} / 2k\Z^{6}}
		\prod_{v=1}^{6}\zeta_{k}^{w_{v}(l_{v}^{2} - 1) / 4}\bigg( \frac{1}{\zeta_{k}^{l_{v}/2} - \zeta_{k}^{-l_{v}/2}} \bigg)^{\deg(v)-2} \nonumber
\\
		&\times \prod_{(v, v') \in \mathrm{Edges}} \frac{\zeta_{k}^{l_{v}l_{v'}/2} - \zeta_{k}^{-l_{v}l_{v'}/2}}{2}.
\label{eq:WRT}
\end{align}
We can write the WRT invariant as a sum of rational functions.

\begin{prop}\label{prop:WRT_inv}
	It holds
	\begin{align*}
		\tau_k(M_3(\Gamma))
		= {}& \frac{{\rm i} \zeta_{k}^{-(18 + \sum_{v=1}^{6}w_{v} + \sum_{v=3}^{6} 1/w_{v})/4}}{4k \big(\zeta_{2k}-\zeta_{2k}^{-1}\big) \sqrt{MN}}
\sum_{\substack{m \in (\Z \smallsetminus k\Z) / 2kw_3w_4\Z \\ n \in (\Z \smallsetminus k\Z) / 2kw_5w_6\Z}}		\zeta_k^{-(m, n) S^{-1} \trans{(m, n)}/4}
		\\
		&\times\frac{\big( \zeta_k^{m/(2w_3)} - \zeta_k^{-m/(2w_3)} \big) \big( \zeta_k^{m/(2w_4)} - \zeta_k^{-m/(2w_4)} \big)}{ \zeta_k^{m/2} - \zeta_k^{-m/2}}
\\
	&{} \times\frac{\big( \zeta_k^{n/(2w_5)} - \zeta_k^{-n/(2w_5)} \big) \big( \zeta_k^{n/(2w_6)} - \zeta_k^{-n/(2w_6)}\big)}{\zeta_k^{n/2} - \zeta_k^{-n/2}},
	\end{align*}
	and thus,\begin{gather*}
\begin{split}
&		\tau_k(M_3(\Gamma))
		=  \frac{{\rm i} \zeta_{k}^{-(18 + \sum_{v=1}^{6}w_{v} + \sum_{v=3}^{6} 1/w_{v})/4}}{4k \big(\zeta_{2k}-\zeta_{2k}^{-1}\big) \sqrt{MN}}
\\
&\hphantom{\tau_k(M_3(\Gamma))=}{}
\times\sum_{\mu = \trans{(m, n)} \in (\Z \smallsetminus k\Z)^2/2kS(\Z^2)}
\zeta_k^{-\trans{\mu} S^{-1} \mu/4}G^\omega\big(\zeta_k^{m/(2M)}\big) G^\varpi\big(\zeta_k^{n/(2N)}\big),
\end{split}
	\end{gather*}
where $ \omega := (w_3, w_4)$, $\varpi := (w_5, w_6) $, and\vspace{-1ex}
	\begin{gather*}
	G^{(w, w')}(z) :=
	\frac{(z^{w} - z^{-w})(z^{w'} - z^{-w'})}{(z^{ww'} - z^{-ww'})}
	=
	- \sum_{n=1}^{\infty} \chi^{(w, w')} (n) z^n.
	\end{gather*}
\end{prop}
\begin{proof}
	By (\ref{eq:WRT}), we obtain
	\begin{align*}
		\tau_{k}(M_3(\Gamma))
		= {}
		&\frac{-{\rm i} \zeta_{k}^{-(18 + \sum_{v=1}^{6}w_{v})/4}}{16k^{3} \big(\zeta_{2k}-\zeta_{2k}^{-1}\big)}
		\sum_{l_{1}, l_{2} \in (\Z \smallsetminus k\Z) / 2k\Z}
		\bigg( \prod_{v=1}^{2} \frac{\zeta_{k}^{w_{v}l_{v}^{2}/4}}{\zeta_{k}^{l_{v}/2} - \zeta_{k}^{-l_{v}/2}} \bigg) \frac{\zeta_{k}^{l_{1}l_{2}/2} - \zeta_{k}^{-l_{1}l_{2}/2}}{2}
\\
		&\times \prod_{v=3}^{4} \sum_{l_{v} \in \Z / 2k\Z} \frac{1}{2} \zeta_{k}^{w_{v}l_{v}^{2}/4} \big( \zeta_{k}^{(l_{1} + 1)l_{v}/2} - \zeta_{k}^{(l_{1} - 1)l_{v}/2} - \zeta_{k}^{(-l_{1} + 1)l_{v}/2 } + \zeta_{k}^{(-l_{1} - 1)l_{v}/2} \big) \\
		&\times \prod_{v=5}^{6} \sum_{l_{v} \in \Z / 2k\Z} \frac{1}{2} \zeta_{k}^{w_{v}l_{v}^{2}/4} \big( \zeta_{k}^{(l_{2} + 1)l_{v}/2} - \zeta_{k}^{(l_{2} - 1)l_{v}/2} - \zeta_{k}^{(-l_{2} + 1)l_{v}/2 } + \zeta_{k}^{(-l_{2} - 1)l_{v}/2} \big).
	\end{align*}
	Applying symmetries $l_{v} \mapsto - l_{v}$, we can show that\vspace{-1ex}
	\begin{align*}
		\tau_{k}(M_3(\Gamma)) = {}
		&\frac{-{\rm i} \zeta_{k}^{-(18 + \sum_{v=1}^{6}w_{v})/4}}{16k^{3} \big(\zeta_{2k}-\zeta_{2k}^{-1}\big)}
		\sum_{l_{1}, l_{2} \in (\Z \smallsetminus k\Z) / 2k\Z}
		\bigg( \prod_{v=1}^{2} \frac{\zeta_{k}^{w_{v}l_{v}^{2}/4}}{\zeta_{k}^{l_{v}/2} - \zeta_{k}^{-l_{v}/2}} \bigg) \frac{\zeta_{k}^{l_{1}l_{2}/2} - \zeta_{k}^{-l_{1}l_{2}/2}}{2}
\\
		&\times \prod_{v=3}^{4} \sum_{l_{v} \in \Z / 2k\Z} \zeta_{k}^{w_{v}l_{v}^{2}/4} \big( \zeta_{k}^{(l_{1} + 1)l_{v}/2} - \zeta_{k}^{(l_{1} - 1)l_{v}/2} \big)
\\
		&\times \prod_{v=5}^{6} \sum_{l_{v} \in \Z / 2k\Z} \zeta_{k}^{w_{v}l_{v}^{2}/4} \big( \zeta_{k}^{(l_{2} + 1)l_{v}/2} - \zeta_{k}^{(l_{2} - 1)l_{v}/2} \big).
	\end{align*}
	holds.
	Applying Proposition \ref{prop:reciprocity} for the case when $ L = \Z $ equipped with the standard inner product $ \sprod{\cdot, \cdot} $ and\vspace{-1ex}
	\begin{gather*}
	(h, z) :=
	\begin{cases}
		(w_{v}, (l_{1} \pm 1)/2k) & \text{if} \ v = 3, 4,
\\
		(w_{v}, (l_{2} \pm 1)/2k) & \text{if} \ v = 5, 6.
	\end{cases}
	\end{gather*}
	we obtain the following equation:
\begin{align*}
		\tau_{k}(M_3(\Gamma))
		= \,
		&\frac{{\rm i} \zeta_{k}^{-(18 + \sum_{v=1}^{6}w_{v})/4}}{4k \big(\zeta_{2k}-\zeta_{2k}^{-1}\big) \prod_{v=3}^{6} \sqrt{\abs{w_{v}}}}
\\
		&\times \sum_{l_{1}, l_{2} \in (\Z \smallsetminus k\Z) / 2k\Z}
		\bigg( \prod_{v=1}^{2} \frac{\zeta_{k}^{w_{v}l_{v}^{2}/4}}{\zeta_{k}^{l_{v}/2} - \zeta_{k}^{-l_{v}/2}} \bigg) \frac{\zeta_{k}^{l_{1}l_{2}/2} - \zeta_{k}^{-l_{1}l_{2}/2}}{2}
\\
		&\times \prod_{v=3}^{4} \sum_{l_{v} \in \Z / w_{v}\Z} \big( \zeta_{k}^{-(2kl_{v} + l_{1} + 1)^{2}/(4w_{v})} - \zeta_{k}^{-(2kl_{v} + l_{1} - 1)^{2}/(4w_{v})} \big)
\\
		&\times \prod_{v=5}^{6} \sum_{l_{v} \in \Z / w_{v}\Z} \big( \zeta_{k}^{-(2kl_{v} + l_{2} + 1)^{2}/(4w_{v})} - \zeta_{k}^{-(2kl_{v} + l_{2} - 1)^{2}/(4w_{v})} \big).
	\end{align*}
	For any map $ f \colon \Z / w_{3}\Z \times \Z / w_{4}\Z \to \bbC $, we have
	\begin{gather*}
	\sum_{\substack{l_{3} \in \Z / w_{3}\Z \\ l_{4} \in \Z / w_{4}\Z}} f(l_3, l_4)
	=	\sum_{n_{1} \in \Z / w_{3}w_{4}\Z} f(n_1, n_1)
	\end{gather*}
	by the Chinese remainder theorem.
	Hence, for any map $ \phi \colon \Z/2k\Z \times \Z / 2k w_{3}\Z \times \Z / 2k w_{4}\Z \to \bbC $, we have
	\begin{gather*}
		\sum_{l_1 \in (\Z \smallsetminus k\Z) / 2k \Z} \sum_{\substack{l_{3} \in \Z / w_{3}\Z \\ l_{4} \in \Z / w_{4}\Z}} \phi (l_1, l_1 + 2k l_3, l_1 + 2k l_4)
\\ \qquad
		{}= \sum_{l_1 \in (\Z \smallsetminus k\Z) / 2k \Z} \sum_{n_{1} \in \Z / w_{3}w_{4}\Z} \phi (l_1, l_1 + 2k n_1, l_1 + 2k n_1)
	\end{gather*}
	by letting $ f(l_3, l_4) := \phi (l_1, l_1 + 2k l_3, l_1 + 2k l_4) $.
	This sum equals
	\begin{gather*}
	\sum_{m \in (\Z \smallsetminus k\Z) / 2k w_3 w_4\Z} \phi (m, m, m)
	\end{gather*}
	by letting $ m := l_1 + 2k n_1 $.
	Similarly, for any map $ \psi \colon \Z/2k\Z \times \Z / 2k w_{5}\Z \times \Z / 2k w_{6}\Z \to \bbC $, we have
	\begin{align*}
		\sum_{l_2 \in (\Z \smallsetminus k\Z) / 2k \Z} \sum_{\substack{l_{5} \in \Z / w_{5}\Z \\ l_{6} \in \Z / w_{6}\Z}} \psi (l_2, l_2 + 2k l_5, l_2 + 2k l_6)
		= 	\sum_{n \in (\Z \smallsetminus k\Z) / 2k w_5 w_6\Z} \psi (n, n, n).
	\end{align*}
	Thus, we obtain
	\begin{align*}
		\tau_k(M_3(\Gamma))		= {}&\frac{{\rm i} \zeta_{k}^{-(18 + \sum_{v=1}^{6}w_{v} + \sum_{v=3}^{6} 1/w_{v})/4}}{8k \big(\zeta_{2k}-\zeta_{2k}^{-1}\big) \sqrt{MN}}
\\
&\times\sum_{\substack{m \in (\Z \smallsetminus k\Z) / 2kw_3w_4\Z \\ n \in (\Z \smallsetminus k\Z) / 2kw_5w_6\Z}}
\big( \zeta_k^{-(m, -n) S^{-1} \trans{(m, -n)}/4} - \zeta_k^{-(m, n) S^{-1} \trans{(m, n)}/4} \big)
\\
&\times\frac{\big( \zeta_k^{m/(2w_3)} - \zeta_k^{-m/(2w_3)} \big)\big( \zeta_k^{m/(2w_4)} - \zeta_k^{-m/(2w_4)} \big)}{ \zeta_k^{m/2} - \zeta_k^{-m/2}}
\\
&\times\frac{\big( \zeta_k^{n/(2w_5)} - \zeta_k^{-n/(2w_5)} \big)\big( \zeta_k^{n/(2w_6)} - \zeta_k^{-n/(2w_6)} \big)}{\zeta_k^{n/2} - \zeta_k^{-n/2}}
	\end{align*}
	since
	\begin{align*}
		-(m, n) S^{-1} \trans{(m, n)}
		&= -c/M m^{2} + 2mn - a/Nn^{2}
\\
		&= \bigg( w_1 - \frac{1}{w_3} - \frac{1}{w_4} \bigg) m^2 + 2mn + \bigg( w_2 - \frac{1}{w_5} - \frac{1}{w_6} \bigg) n^2.
	\end{align*}
	By replacing $ n \mapsto -n $, we have
	\begin{gather*}
		\sum_{ n \in (\Z \smallsetminus k\Z) / 2kw_5w_6\Z }
		\zeta_k^{-(m, -n) S^{-1} \trans{(m, -n)}/4}
		\frac{\big( \zeta_k^{n/(2w_5)} - \zeta_k^{-n/(2w_5)} \big)\big( \zeta_k^{n/(2w_6)} - \zeta_k^{-n/(2w_6)} \big)}
		{\zeta_k^{n/2} - \zeta_k^{-n/2}}
\\ \qquad
		{}= -\sum_{ n \in (\Z \smallsetminus k\Z) / 2kw_5w_6\Z }
		\zeta_k^{-(m, n) S^{-1} \trans{(m, n)}/4}
		\frac{\big( \zeta_k^{n/(2w_5)} - \zeta_k^{-n/(2w_5)} \big)\big( \zeta_k^{n/(2w_6)} - \zeta_k^{-n/(2w_6)} \big)}
		{\zeta_k^{n/2} - \zeta_k^{-n/2}}.
	\end{gather*}
	Therefore we prove the claim.
\end{proof}

Then we calculate the asymptotic expansion of $G^\omega\big(\zeta_k^{m/(2M)} {\rm e}^{-tK/(2M)}\big) G^\varpi\big(\zeta_k^{n/(2N)} {\rm e}^{-tL/(2N)}\big)$.

\begin{lem} \label{lem:G(z)_asymptotic}
	Let $ K $ and $ L$ be positive real numbers and $\mu = \trans{(m, n)}$ be a pair of integers.
	For sufficiently small $t > 0$,
	\begin{gather*}
		G^\omega\big(\zeta_k^{m/(2M)} {\rm e}^{-tK/(2M)}\big) G^\varpi\big(\zeta_k^{n/(2N)} {\rm e}^{-tL/(2N)}\big)
\\ \qquad
	{}	= 		\sum_{\gamma = \trans{(\alpha, \beta)} \in (2S)^{-1}(\Z^2)/k\Z^2} \veps(\gamma) \bm{e}\bigg(\frac{1}{k} {\trans{\gamma}} \mu\bigg)
		\sum_{j,l = -1}^\infty \frac{B_{j+1}(\alpha/k) B_{l+1}(\beta/k)}{(j+1)! (l+1)!} K^j L^l (-kt)^{j+l}
	\end{gather*}
	holds.
\end{lem}
\begin{proof}
	By the definitions of $G^\omega$ and $G^\varpi$, the left hand side of the above equation equals to
	\begin{gather*}
	\sum_{m', n' = 0}^\infty \chi^\omega(m') \chi^\varpi(n')
	\big(\zeta_k^{m} {\rm e}^{-tK}\big)^{m'/2M} \big(\zeta_k^{n} {\rm e}^{-tL}\big)^{n'/2N}.
	\end{gather*}
	Letting $ \gamma = \trans{(m'/(2M), n'/(2N))} $ leads to
	\begin{gather*}
	\sum_{0 \le \gamma \in (2S)^{-1}(\Z^2)} \veps(\gamma) \bm{e}\bigg( \frac{1}{k} \trans{\gamma} \mu \bigg)
	{\rm e}^{-t (K, L) \gamma}
	\end{gather*}
	and the definition of Bernoulli polynomial shows the lemma.
\end{proof}

The constant term in the above expansion is simplified in the following lemma.

\begin{prop} \label{prop:G(z)_value}
	For $ \mu = \trans{(m, n)} \in (\Z \smallsetminus k\Z)^2 $
	\begin{gather*}
	G^\omega\big(\zeta_k^{m/(2M)}\big) G^\varpi\big(\zeta_k^{n/(2N)}\big)
	= \frac{1}{k^2} \sum_{ \gamma = \trans{(\alpha, \beta)} \in (2S)^{-1}(\Z^2) \cap [0, k)^2}
	\veps(\gamma) \bm{e}\bigg( \frac{1}{k} \trans{\gamma} \mu \bigg) \alpha \beta
	\end{gather*}
	holds.
\end{prop}
\begin{proof}
	Lemma \ref{lem:G(z)_asymptotic} shows that for any $ K, L > 0 $
	\begin{align*}
		G^\omega\big(\zeta_k^{m/(2M)}\big) G^\varpi\big(\zeta_k^{n/(2N)}\big)
	= {}&\lim_{t \to 0} G^\omega\big(\zeta_k^{m/(2M)} {\rm e}^{-tK/(2M)}\big) G^\varpi\big(\zeta_k^{n/(2N)} {\rm e}^{-tL/(2N)}\big)
\\
		= {}&\sum_{\gamma = \trans{(\alpha, \beta)} \in (2S)^{-1}(\Z^2) \cap [0, k)^2} \veps(\gamma) \bm{e}\bigg(\frac{1}{k} \trans{\gamma} \mu \bigg)
\\
	&\times	\bigg(
		\frac{K}{2L} B_2 \bigg( \frac{\alpha}{k} \bigg)
		+ \frac{L}{2K} B_2 \bigg( \frac{\beta}{k} \bigg)
		+ B_1 \bigg( \frac{\alpha}{k} \bigg) B_1 \bigg( \frac{\beta}{k} \bigg)
		\bigg)
	\end{align*}
	holds.
	The claim follows from the following lemma.
\end{proof}
\begin{lem} 
	For any $ \mu = \trans{(m, n)} \in (\Z \smallsetminus k\Z) \times \Z $ and any map $ B \colon \Q \to \bbC $, we have
	\begin{gather*}
	\sum_{\gamma = \trans{(\alpha, \beta)} \in (2S)^{-1}(\Z^2) \cap [0, k)^2}
	\veps(\gamma) \bm{e}\bigg( \frac{1}{k} \trans{\gamma} \mu \bigg) B (\beta)
	= 0.
	\end{gather*}
\end{lem}
\begin{proof}
	The left hand side can be written as
	\begin{gather*}
	\sum_{\beta \in \frac{1}{2N}\Z/k\Z} \chi^\varpi(\beta) \bm{e}\bigg( \frac{\beta n}{k} \bigg) B (\beta)
	\sum_{\alpha \in \frac{1}{2M}\Z/\Z} \chi^\omega(\beta) \bm{e}\bigg( \frac{\alpha m}{k} \bigg)
	\sum_{l \in \Z/k\Z} \bm{e}\bigg( \frac{ml}{k} \bigg).
	\end{gather*}
	Since $ k \nmid m $, the last sum vanishes.
	Thus, we obtain the claim.
\end{proof}
Propositions \ref{prop:WRT_inv} and~\ref{prop:G(z)_value} assert that the following proposition holds.
\begin{prop}
	\begin{align*}
		\tau_k(M_3(\Gamma)) = \, \frac{\zeta_{k}^{-(18 + \sum_{v=1}^{6}w_{v} + \sum_{v=3}^{6} 1/w_{v})/4}}{2 \big(\zeta_{2k}-\zeta_{2k}^{-1}\big)}
		\frac{1}{k^{2}} \sum_{ \gamma = \trans{(\alpha, \beta)} \in (2S)^{-1}(\Z^2) \cap [0, k)^2}
		\veps(\gamma) \bm{e}\bigg( \frac{1}{k} Q(\gamma) \bigg) \alpha \beta.
	\end{align*}
\end{prop}
\begin{proof}
	Propositions \ref{prop:Gauss_sum}$(v)$, \ref{prop:WRT_inv} and~\ref{prop:G(z)_value} show
	\begin{align*}
		\tau_k(M_3(\Gamma)) = {} &\frac{{\rm i} \zeta_{k}^{-(18 + \sum_{v=1}^{6}w_{v} + \sum_{v=3}^{6} 1/w_{v})/4}}{4k \big(\zeta_{2k}-\zeta_{2k}^{-1}\big) \sqrt{MN}}
		\frac{1}{k^2} \sum_{\mu = \trans{(m, n)} \in \Z^2/2kS(\Z^2)}
		\zeta_k^{-\trans{\mu} S^{-1} \mu/4}
\\
		&\times\sum_{\gamma = \trans{(\alpha, \beta)} \in (2S)^{-1}(\Z^2) \cap [0, k)^2}
		\veps(\gamma) \bm{e}\bigg( \frac{1}{k} \trans{\gamma} \mu \bigg) \alpha \beta.
	\end{align*}
	Replacing $ \mu $ with $ \mu + 2S\gamma $, we obtain
	\begin{align*}
	\tau_k(M_3(\Gamma)) = {}&\frac{{\rm i} \zeta_{k}^{-(18 + \sum_{v=1}^{6}w_{v} + \sum_{v=3}^{6} 1/w_{v})/4}}{4k \big(\zeta_{2k}-\zeta_{2k}^{-1}\big) \sqrt{MN}} G(2kS)
	\frac{1}{k^2}
\\
&\times\sum_{ \gamma = \trans{(\alpha, \beta)} \in (2S)^{-1}(\Z^2) \cap [0, k)^2}
	\veps(\gamma) \bm{e}\bigg( \frac{1}{k} Q(\gamma) \bigg) \alpha \beta,
	\end{align*}
	where
	\begin{gather*}
	G(2kS) := \sum_{\mu = \trans{(m, n)} \in \Z^2/2kS(\Z^2)} \bm{e}\bigg({-}\frac{1}{2} \trans{\mu} (2kS)^{-1} \mu \bigg).
	\end{gather*}
	By Proposition \ref{prop:reciprocity}, $ G(2S) = -2k{\rm i} \sqrt{MN} $ and this completes the proof.
\end{proof}

\subsection*{Acknowledgements}

The first and second author are supported by JSPS KAKENHI Grant Number JP 21J10271 and 20J20308.
The first author was supported by a Scholarship of Tohoku University, Division for Interdisciplinary Advanced Research and Education.
We would like to show their greatest appreciation to Professor Yuji Terashima and Takuya Yamauchi for giving many pieces of advice.
We are deeply grateful to Professor Kazuhiro Hikami for giving many comments.
The first author thanks his family for all the support.
We thank the referees for helpful suggestions and comments which substantially improved the presentation of our paper.

\pdfbookmark[1]{References}{ref}
\LastPageEnding


\begin{thebibliography}{99}
\footnotesize\itemsep=0pt

\bibitem{AM}
Andersen J.E., Misteg{\aa}rd W.E., Resurgence analysis of quantum invariants of
 {S}eifert fibered homology spheres, \href{https://doi.org/10.1112/jlms.12506}{\textit{J.~Lond. Math. Soc.}} \textbf{105}
 (2022), 709--764, \href{https://arxiv.org/abs/1811.05376}{arXiv:1811.05376}.

\bibitem{BKM}
Bringmann K., Kaszian J., Milas A., Higher depth quantum modular forms,
 multiple {E}ichler integrals, and {$\mathfrak{sl}_3$} false theta functions,
 \href{https://doi.org/10.1007/s40687-019-0182-4}{\textit{Res. Math. Sci.}} \textbf{6} (2019), 20, 41~pages,
 \href{https://arxiv.org/abs/1704.06891}{arXiv:1704.06891}.

\bibitem{BMM}
Bringmann K., Mahlburg K., Milas A., Higher depth quantum modular forms and
 plumbed 3-manifolds, \href{https://doi.org/10.1007/s11005-020-01310-z}{\textit{Lett. Math. Phys.}} \textbf{110} (2020),
 2675--2702, \href{https://arxiv.org/abs/1906.10722}{arXiv:1906.10722}.

\bibitem{BM}
Bringmann K., Milas A., {$\mathcal W$}-algebras, false theta functions and
 quantum modular forms,~{I}, \href{https://doi.org/10.1093/imrn/rnv033}{\textit{Int. Math. Res. Not.}} \textbf{2015}
 (2015), 11351--11387.

\bibitem{CCFGH}
Cheng M.C., Chun S., Ferrari F., Gukov S., Harrison S.M., 3d modularity,
 \href{https://doi.org/10.1007/jhep10(2019)010}{\textit{J.~High Energy Phys.}} \textbf{2019} (2019), no.~10, 010, 93~pages,
 \href{https://arxiv.org/abs/1809.10148}{arXiv:1809.10148}.

\bibitem{Ch}
Chun S., A resurgence analysis of the {${\rm SU}(2)$} {C}hern--{S}imons partition
 functions on a {B}rieskorn homology sphere {$\Sigma(2,5,7)$}, \href{https://arxiv.org/abs/1701.03528}{arXiv:1701.03528}.

\bibitem{DT}
Deloup F., Turaev V., On reciprocity, \href{https://doi.org/10.1016/j.jpaa.2005.12.008}{\textit{J.~Pure Appl. Algebra}}
 \textbf{208} (2007), 153--158, \href{https://arxiv.org/abs/math.AC/0512050}{arXiv:math.AC/0512050}.

\bibitem{FIMT}
Fuji H., Iwaki K., Murakami H., Terashima Y., Witten--{R}eshetikhin--{T}uraev
 function for a knot in {S}eifert manifolds, \href{https://doi.org/10.1007/s00220-021-03953-y}{\textit{Comm. Math. Phys.}}
 \textbf{386} (2021), 225--251, \href{https://arxiv.org/abs/2007.15872}{arXiv:2007.15872}.

\bibitem{GMP}
Gukov S., Marino M., Putrov P., {R}esurgence in complex {C}hern--{S}imons
 theory, \href{https://arxiv.org/abs/1605.07615}{arXiv:1605.07615}.

\bibitem{GPPV}
Gukov S., Pei D., Putrov P., Vafa C., B{PS} spectra and 3-manifold invariants,
 \href{https://doi.org/10.1142/S0218216520400039}{\textit{J.~Knot Theory Ramifications}} \textbf{29} (2020), 2040003, 85~pages,
 \href{https://arxiv.org/abs/1701.06567}{arXiv:1701.06567}.

\bibitem{H_Bries}
Hikami K., On the quantum invariant for the {B}rieskorn homology spheres,
 \href{https://doi.org/10.1142/S0129167X05003004}{\textit{Internat.~J. Math.}} \textbf{16} (2005), 661--685,
 \href{https://arxiv.org/abs/math-ph/0405028}{arXiv:math-ph/0405028}.

\bibitem{H_Lattice}
Hikami K., Quantum invariant, modular form, and lattice points, \href{https://doi.org/10.1155/IMRN.2005.121}{\textit{Int.
 Math. Res. Not.}} \textbf{2005} (2005), 121--154, \href{https://arxiv.org/abs/math-ph/0409016}{arXiv:math-ph/0409016}.

\bibitem{H_Lattice2}
Hikami K., Quantum invariants, modular forms, and lattice points.~{II},
 \href{https://doi.org/10.1063/1.2349484}{\textit{J.~Math. Phys.}} \textbf{47} (2006), 102301, 32~pages,
 \href{https://arxiv.org/abs/math.QA/0604091}{arXiv:math.QA/0604091}.

\bibitem{H_Seifert}
Hikami K., On the quantum invariants for the spherical {S}eifert manifolds,
 \href{https://doi.org/10.1007/s00220-006-0094-1}{\textit{Comm. Math. Phys.}} \textbf{268} (2006), 285--319,
 \href{https://arxiv.org/abs/math-ph/0504082}{arXiv:math-ph/0504082}.

\bibitem{LR}
Lawrence R., Rozansky L., Witten--{R}eshetikhin--{T}uraev invariants of
 {S}eifert manifolds, \href{https://doi.org/10.1007/s002200050678}{\textit{Comm. Math. Phys.}} \textbf{205} (1999),
 287--314.

\bibitem{LZ}
Lawrence R., Zagier D., Modular forms and quantum invariants of
 {$3$}-manifolds, \href{https://doi.org/10.4310/AJM.1999.v3.n1.a5}{\textit{Asian~J. Math.}} \textbf{3} (1999), 93--107.

\bibitem{RT}
Reshetikhin N., Turaev V.G., Invariants of {$3$}-manifolds via link polynomials
 and quantum groups, \href{https://doi.org/10.1007/BF01239527}{\textit{Invent. Math.}} \textbf{103} (1991), 547--597.

\bibitem{R}
Turaev V.G., Quantum invariants of knots and 3-manifolds, \href{https://doi.org/10.1515/9783110435221}{\textit{De Gruyter
 Studies in Mathematics}}, Vol.~18, De Gruyter, Berlin, 2016.

\bibitem{W}
Wu D.H., Resurgent analysis of {$\rm SU(2)$} {C}hern--{S}imons partition
 function on {B}rieskorn spheres {$\Sigma(2, 3, 6n + 5)$}, \href{https://doi.org/10.1007/jhep02(2021)008}{\textit{J.~High
 Energy Phys.}} \textbf{2021} (2021), no.~2, 008, 18~pages,
 \href{https://arxiv.org/abs/2010.13736}{arXiv:2010.13736}.

\bibitem{Zagier_quantum}
Zagier D., Quantum modular forms, in Quanta of Maths, \textit{Clay Math.
 Proc.}, Vol.~11, Amer. Math. Soc., Providence, RI, 2010, 659--675.

\end{thebibliography}
\end{document}